\documentclass[a4paper,12pt]{amsart}
\usepackage[utf8]{inputenc}
\usepackage[T1]{fontenc}
\usepackage[UKenglish]{babel}
\usepackage[margin=18mm]{geometry}

\usepackage{color}%cyan,red;magenta,green;yellow,blue

\usepackage{graphicx}

\usepackage{amsmath,amssymb,amsfonts,amsthm}
\usepackage{mathrsfs,eucal,dsfont}
\usepackage{verbatim,enumitem}

\usepackage{hyperref,url}

\newcommand{\R}{\mathds R}

\newcommand{\I}{\mathds 1}
\newcommand{\var}{\operatorname{Var}}

\def\aa{\alpha}
\def\dd{\delta}
\def\d{{\rm d}}
\def\<{\langle}
\def\>{\rangle}

 \def\ff{\frac}
 \def\ss{\sqrt}
\def\bb{\beta}

\def\R{\mathbb R}  \def\ff{\frac} \def\ss{\sqrt} 
 \def\kk{\kappa} 
\def\dd{\delta}  \def\vv{\varepsilon} 
\def\<{\langle} \def\>{\rangle}  
  \def\nn{\nabla}  
\def\d{\text{\rm{d}}} \def\bb{\beta} \def\aa{\alpha} 
  \def\si{\sigma} 
 \def\beq{\begin{equation}}  
 
\def\e{\text{\rm{e}}}

  \def\ll{\lambda}

\def\to{\rightarrow}
\def\8{\infty}\def\3{\triangle}
\def\1{\lesssim}

\renewcommand{\bar}{\overline}
\renewcommand{\hat}{\widehat}
\renewcommand{\tilde}{\widetilde}

\newtheorem{theorem}{Theorem}[section]
\newtheorem{lemma}[theorem]{Lemma}
\newtheorem{proposition}[theorem]{Proposition}
\newtheorem{corollary}[theorem]{Corollary}

\theoremstyle{definition}

\newtheorem{remark}[theorem]{Remark}

\numberwithin{equation}{section}
\begin{document}
\allowdisplaybreaks

\title[Stochastic Hamiltonian systems with L\'evy noises] {Coupling approach for exponential ergodicity of stochastic Hamiltonian systems with L\'evy noises}

\author{
Jianhai Bao\qquad\,
\and\qquad
Jian Wang}
\date{}
\thanks{\emph{J.\ Bao:} Center for Applied Mathematics, Tianjin University, 300072  Tianjin, P.R. China. \url{jianhaibao@tju.edu.cn}}
\thanks{\emph{J.\ Wang:}
College of Mathematics and Informatics \&  Fujian Key Laboratory of Mathematical
Analysis and Applications (FJKLMAA) \&  Center for Applied Mathematics of Fujian Province (FJNU), Fujian Normal University, 350007 Fuzhou, P.R. China. \url{jianwang@fjnu.edu.cn}}

\maketitle

\begin{abstract}
 We establish exponential ergodicity for the stochastic Hamiltonian system $(X_t, V_t)_{t\ge0}$ on $\R^{2d}$ with L\'evy noises
\begin{align*}
\begin{cases}
\d X_t=\big(a X_t+bV_t\big)\,\d t,\\
\d V_t=U(X_t,V_t)\,\d t+\d L_t,
\end{cases}
\end{align*}
where $a\ge 0$, $b> 0$, $U:\R^{2d}\to\R^d$ and $(L_t)_{t\ge0}$ is an $\R^d$-valued pure jump  L\'{e}vy process.
The approach is based on a new refined basic coupling for L\'evy processes and a Lyapunov function for stochastic Hamiltonian systems.
In particular, we can handle the case that
$U(x,v)=-v-\nabla U_0(x)$ with double well potential $U_0$ which is super-linear growth at infinity such as
$U_0(x)=c_1(1+|x|^2)^l-c_2|x|^2$ with $l>1$ or $U_0(x)=c_1\e^{(1+|x|^2)^l}-c_2|x|^2$ with $l>0$ for any $c_1,c_2>0$, and also deal with the case that the L\'evy measure $\nu$ of $(L_t)_{t\ge0}$ is degenerate in the sense that $$\nu(\d z)\ge \frac{c}{|z|^{d+\theta_0}}
\I_{\{0<z_1\le 1\}}\,\d z$$  for some $c>0$ and $\theta_0\in (0,2)$, where $z_1$ is the first component of the vector $z\in \R^d$.

\medskip


\end{abstract}

\section{Introduction and main result}\label{section1}
The kinetic Langevin diffusion $(X_t,V_t)_{t\ge0}$ on $\R^{2d}:=\R^d\times \R^d$, which describes the motion of a particle with position $X_t$ and velocity $V_t$ in the statistical physics, is given by
\begin{equation}\label{E1aa}
\begin{cases}
\d X_t=V_t\,\d t,\\
\d V_t=- V_t\,\d t-\nn U(X_t)\,\d t+\d B_t,
\end{cases}
\end{equation} where $U\in
C^{1}(\R^d)$ and $(B_t)_{t\ge0}$ is a $d$-dimensional Brownian motion. In particular, $-\nabla U(x)$ means the force subject to damping and random collisions. Since the driven noise appears only in the second component, \eqref{E1aa} is a degenerate stochastic differential equation (SDE). This type degenerate SDE is named as a  stochastic damping Hamiltonian system with the Hamiltonian function $$H(x,v):=U(x)+\ff{|v|^2}{2}$$ in the probability community; see \cite{Wu}.
When $U$ is smooth, the celebrated H\"{o}rmander's hypoellipticity theorem tells us that the distribution density $p(t,x,v)$ of the process $(X_t,V_t)_{t\ge0}$ given by \eqref{E1aa} is also smooth and it solves the following kinetic Fokker-Planck equation
\begin{equation}\label{e:eeFP}\partial_t p+v\cdot \nabla_x p-\nabla U(x)\cdot \nabla_v p=\frac{1}{2} \triangle_vp-{\rm div}_v(vp).\end{equation}
For background on stochastic Hamiltonian systems and related kinetic Fokker-Planck equations as well as their applications, the readers are referred to \cite{Soize}. Furthermore, one can verify that the measure
$$\mu_*(\d x,\d v):=\e^{-H(x,v)}\,\d x\,\d v$$ is invariant for the process $(X_t,V_t)_{t\ge0}$. Equivalently, the equilibrium of the kinetic Fokker-Planck equation \eqref{e:eeFP} is given by
$$\rho_\infty(x,v)=\e^{-H(x,v)}.$$ Nowadays, the rate of convergence to the equilibrium
of \eqref{E1aa} has been investigated considerably; see e.g.\ \cite{Vi}.

In this work, we will consider \eqref{E1aa} with the Brownian motion $(B_t)_{t\ge0}$ therein replaced by a pure jump L\'evy process $(L_t)_{t\ge0}$. More generally, we will consider the
following stochastic Hamiltonian system driven by L\'evy noises:
\begin{equation}\label{EE1}
\begin{cases}
\d X_t=(a X_t+bV_t)\,\d t,\\
\d V_t=U(X_t,V_t)\,\d t+\d L_t,
\end{cases}
\end{equation}
where $a\ge0$, $b>0$, $U:\R^{2d}\to\R^d$  and $(L_t)_{t\ge0}$ is an $\R^d$-valued pure jump L\'evy process associated with the L\'{e}vy measure $\nu$ on $(\R^d, \mathscr{B}(\R^d))$ such that $\nu(\{{\bf0}\})=0$ and $\int_{\R^d}(1\wedge|u|^2)\,\nu(\d u)<\8$;  that is, the infinitesimal generator $\mathscr L_0$ for the process $(L_t)_{t\ge0}$ is given by
\begin{equation}\label{EE0}
(\mathscr L_0 g)(x)=\int_{\R^d}\big(g(x+u)-g(x)-\<\nn g(x),u\>\I_{\{|u|\le 1\}}\big)\,\nu(\d u),\quad g\in C_b^2(\R^d).
\end{equation}Analogous to \eqref{E1aa}, \eqref{EE1} can be used to describe  the motion of particles perturbed by a discontinuous stochastic force, and \eqref{EE1} has a close connection with the non-local kinetic Fokker-Planck equation. Recently the study of stochastic Hamiltonian systems driven by L\'evy noises and non-local kinetic Fokker-Planck equations has been developed greatly. For example, see \cite{DPSZ,Zhangb} on the existence of smooth fundamental solutions for \eqref{EE1} when $(L_t)_{t\ge0}$ is a subordinated Brownian motion, and see \cite{CHM,CZ,HWZ,HMP} on Schauder estimates and $L^p$-estimates for non-local kinetic Fokker-Planck equations, and so on. However, the research on the ergodicity of stochastic Hamiltonian systems driven by L\'evy noises (equivalently, the convergence to equilibrium for non-local kinetic Fokker-Planck equations) is still widely open. The purpose of this paper is to study exponential ergodicity under a proper Wasserstein-type distance for the stochastic Hamiltonian system with L\'evy noises \eqref{EE1}.

\ \

To present our main result, we need to state assumptions on $(X_t,V_t)_{t\ge0}$ given by \eqref{EE1}. First, we assume that $U:\R^{2d}\to\R^d$ satisfies the condition as follows:
\begin{enumerate}
\item[$({\bf A_0})$] \it
 The function $(x,v)\mapsto U(x,v)$ is locally Lipschitz continuous, i.e., for any $R>0$, there is a constant $\ll(R)>0$ such that for all $x,x',v,v'\in B(0,R):=\{z\in \R^d: |z|\le R\},$
$$
\big|U(x,v)-U(x',v')|\le \ll(R)(|x-x'|+|v-v'|).
$$
\end{enumerate}
Then, it is well known that the SDE \eqref{EE1} has a unique strong solution $(X_t, V_t)_{t\ge0}$ up to the explosive time
\begin{equation}\label{e:exp}\xi:=\inf\{t>0: |X_t|+|V_t|=\infty\};\end{equation} see \cite[Theorem 6.2.3, p.\ 367]{AP} for more details.
Moreover, the process $(X_t, V_t)_{t\ge0}$ has the strong Markov property, and the associated generator, acting on $f\in C_b^{2}(\R^{2d})$, is given by
\begin{equation} \label{EE-}
\begin{split}
(\mathscr L f)(x,v)&=\big\<ax+bv,\nn_x f(x,v) \big\>+\big\<U(x,v),\nn_v f(x,v)\big\>\\
&\quad+\int_{\R^d}\big(f(x,v+ u)-f(x,v)- \<\nn_v f(x,v),u\>\I_{\{|u|\le 1\}}\big)\,\nu(\d u).
\end{split}
\end{equation}

Next, we further impose the following two assumptions.
\begin{itemize}\it
\item[$({\bf A_1})$] There exist a $C^{1,2}$-function $\mathcal W:\R^{2d} \to[1,\8)$ with  $\mathcal W(x,v)\to\8$ as $|x|+|v|\to\8$ and constants $c_0,C_0>0$ such that for all $x,v\in\R^d,$
\begin{equation}\label{E:lya}
(\mathscr L \mathcal W)(x,v)\le -c_0 \mathcal W(x,v)+C_0,
\end{equation} where $\mathscr L $ is given by \eqref{EE-} and
$C^{1,2}(\R^{2d})$ is the set of all real-valued functions $f(x,v)$ defined on
$\R^d\times \R^d$ which are continuously once differentiable in $x\in \R^d$ and twice differentiable in $v\in \R^d$.
\item[$({\bf A_2})$]
There is
 a non-negative measure $\nu^*\le \nu$ on $(\R^d,\mathscr{B}(\R^d))$ such that
\begin{itemize}
\item[{\rm (i)}]  there exist a constant $r_0>0$ and a non-decreasing and concave function $\si_{r_0}(\cdot)\in C([0,r_0])\cap C^2((0,r_0])$ such that $
\int_0^{r_0}
\si_{r_0} (l)\,\d l<\infty$ and
$ \si_{r_0}(r)\le rJ(r)$ for all $r\in (0,r_0]$, where \begin{equation}\label{e:func-nv}J(s):=\inf_{|x|\le s} \big(\nu^*\wedge (\delta_x* \nu^*)\big)(\R^d),\quad s>0.\end{equation}
\item[{\rm (ii)}] there exist constants $\eta\in(0,1)$ and $c_*>0$ such that for all $x,v\in \R^d$,
$$
\int_{\R^d}\big|\mathcal W(x,v+u)- \mathcal W(x,v)\big|\,\nu^*(\d u)\le c_*\mathcal W(x,v)^{\eta},
$$ where $\mathcal W$ is given in ${\bf (A_1)}$.
\end{itemize}
\end{itemize}

${\bf(A_1)}$ indicates that $\mathcal W(x,v)$ is a Lyapunov function for the process $(X_t,V_t)_{t\ge0}$. In particular, according to \eqref{E:lya} and \cite[Theorem 1.2]{MT}, the process $(X_t,V_t)_{t\ge0}$ is conservative; that is, the explosive time defined by \eqref{e:exp} satisfies that  $\xi=\infty$ a.s. Roughly speaking, ${\bf(A_2)}$(i) means that there are many active small jumps of the L\'evy process $(L_t)_{t\ge0}$, which can be regarded as a non-degenerate condition for the L\'evy measure near zero; see \cite{LMW,LW} for more details. ${\bf(A_2)}$(ii) is a technical condition to handle the exponential ergodicity in terms of multiplicative Wasserstein-type distance as shown in Theorem \ref{T:main1}.

In the following, let $P_t((x,v),\cdot)$ be the transition kernel of the process $(X_t,V_t) $ with the starting point $(x,v)$.

\begin{theorem}\label{T:main1}Under Assumptions ${\bf(A_0)}$, ${\bf (A_1)}$ and ${\bf(A_2)}$, the process $(X_t, V_t)_{t\ge0}$ determined by \eqref{EE1} is exponentially ergodic in the sense that there exist  a unique invariant probability measure $\mu$ and a constant $\lambda_*>0$ such that for all $x,v\in \R^d$ and $t>0$,
\begin{equation}\label{e:main-1}W_{\Psi}\big(P_t((x,v),\cdot),\mu\big)\le C(x,v)\e^{-\lambda_*t},\end{equation}where \begin{equation}\label{e:main-2}\Psi\big((x,v),(x',v')\big):=\big((|x-x'|+|v-v'|)\wedge 1\big)\big(\mathcal W(x,v)+\mathcal W(x',v')\big)\end{equation} and $C(x,v)$ is a positive measurable function dependent on $x,v$. \end{theorem}

For the definition of the Wasserstein-type distance $W_\Psi$,  one can refer to Subsection \ref{Appendix1} in the Appendix section.

\subsection{Example:  kinetic Langevin process with L\'evy noises} As an application of Theorem \ref{T:main1}, we are going to treat  exponential ergodicity under the Wasserstein-type  distance for
the following kinetic Langevin process $(X_t,V_t)_{t\ge0}$ on $\R^{2d}$
\begin{equation}\label{E1}
\begin{cases}
\d X_t=V_t\,\d t,\\
\d V_t=-\aa V_t\,\d t-\bb \nn U_0(X_t)\,\d t+\d L_t,
\end{cases}
\end{equation}
where $\aa,\bb>0$, $U_0\in C^1(\R^d)$ and $(L_t)_{t\ge0}$ is an $\R^d$-valued pure jump  L\'{e}vy process with the L\'{e}vy measure $\nu$ satisfying that  for some $\theta\in (0,1]$, \begin{equation}\label{e:lll}\int_{\R^d}\big(|u|^2\wedge |u|^\theta\big)\,\nu(\d u)<\8.\end{equation} That is, \eqref{E1} is a special case of the SDE \eqref{EE1} with $a=0$, $b=1$ and $U(x,v)=-\aa v-\bb \nabla U_0(x).$
Note that  \eqref{E1} is exactly the same form as \eqref{E1aa} with $(B_t)_{t\ge0}$ replaced by $(L_t)_{t\ge0}$.

Below we will assume that the potential term $U_0:\R^d\to \R$ fulfills the following conditions:
 \begin{enumerate}
\item[$({\bf B_0})$] \it  $x\mapsto \nn U_0(x)$ is locally Lipschitz continuous, i.e., for any $R>0,$  there exists a constant $\lambda_{U_0}(R)>0$ such that
$$
|\nn U_0(x)-\nn U_0(y)|\le \lambda_{U_0}(R) |x-y|,\quad x,y\in B(0,R).$$

\item[$({\bf B_1})$] There exist constants $\ll_1>0$ and $\ll_i\ge 0$ $(i=2,3,4,5)$ with
\begin{equation}\label{W3}
 \ll_2\ll_4<\ll_1,\quad 2\bb\ll_4 \le \ff{\aa^2}{4} +\ss{\bb(\ll_1-\ll_2\ll_4)}\, \aa
\end{equation}
 such that
\begin{equation}\label{E3---1}
\big\<x,\nn U_0(x)\big\>\ge \ll_1|x|^2+\ll_2U_0(x)-\ll_3,\quad x\in\R^d
\end{equation}
and
\begin{equation}\label{E3---}
U_0(x)\ge -\ll_4|x|^2-\ll _5,\quad x\in\R^d.
\end{equation}
\item[${\bf(B_2)}$] There are constants $c>0$ and $\theta_0\in (0,\theta/2)$ such that
\begin{equation}\label{K-}\nu(\d z)\ge \frac{c}{|z|^{d+\theta_0}}
\I_{\{0<z_1\le 1\}}\,\d z,\end{equation} where $z_1$ is the first component of the vector $z\in \R^d$.
\end{enumerate}

\begin{remark} \label{R:1.3} We make some comments on Assumption $({\bf B_1})$.
\begin{itemize}
\item[(i)] The condition $\ll_2\ll_4<\ll_1$ is reasonable, since it, together with \eqref{E3---1} and \eqref{E3---}, gives us that for all $x\in \R^d$,
\begin{equation}\label{E3---2}\langle x,\nabla U_0(x)\rangle\ge (\ll_1-\ll_2\ll_4) |x|^2-\ll_2\ll_5-\ll_3\end{equation} with $\ll_1>\ll_2\ll_4$. \eqref{E3---2} is a common condition in the literature to yield the exponential ergodicity of Langevin diffusions and SDEs with additive L\'evy noise; see \cite{EGZb, LMW}.

\item[(ii)] When $U_0$ is bounded from below by a negative constant, we can take $\ll_4=0$ and so \eqref{W3} and \eqref{E3---} hold trivially. On the other hand, it is easy to see that \eqref{E3---1} holds for $U_0(x)=(1+|x|^2)^l$ with $l\ge1$ or $U_0(x)=\e^{(1+|x|^2)^l}$ with $l>0$, which is different from assumptions of the potential $U_0(x)$ for kinetic Langevin diffusions in \cite{EGZ,Vi}.  More precisely,
    in \cite[Theorem A.\ 8]{Vi} it was required that $|\nabla^2U_0|\le c(1+|U_0|)$ for some $c>0$; in \cite[Assumption 2.1]{EGZ} it was assumed that $\nn U_0(x)$ is globally Lipschitz continuous.

\item[(iii)] We can further check that  ${\bf(B_1)}$ indeed holds for a large class of double well potentials which are super-linear growth at infinity, including $U_0(x)=c_1(1+|x|^2)^l-c_2|x|^2$ with $l>1$ or $U_0(x)=c_1\e^{(1+|x|^2)^l}-c_2|x|^2$ with $l>0$ for any $c_1,c_2>0$. See Subsection \ref{section5.3} in the Appendix section of this paper for the simple proof.
\end{itemize}\end{remark}

As we shall see  in Section \ref{section4},
for the framework \eqref{E1},
${\bf(B_1)}$ and \eqref{e:lll} are imposed herein to guarantee that the Lyapunov condition ${\bf (A_1)}$ is valid, while ${\bf(B_2)}$ is put to ensure that ${\bf (A_2)}$ is satisfied under \eqref{e:lll}. It is clear that, for the SDE \eqref{E1}, Assumption ${\bf(A_0)}$ holds true under Assumption ${\bf(B_0)}$.  Therefore, we have
\begin{theorem}\label{T:main2} Under Assumptions ${\bf(B_0)}$, ${\bf(B_1)}$ and ${\bf(B_2)}$, the assertion of Theorem $\ref{T:main1}$ holds true for the process $(X_t,V_t)_{t\ge0}$ defined by \eqref{E1}, where
$\Psi((x,v),(x',v'))$ is given by \eqref{e:main-2} with
$$W(x,v)=\big(1+W_0(x) +|x|^2+|v|^2\big)^{\theta/2},\quad x,v\in \R^d,$$ where $\theta\in (0,1]$ is given in \eqref{e:lll} and
$W_0(x):=U_0(x) + \ll_4|x|^2+\ll_5$ with constants $\ll_4,\ll_5$ given in \eqref{E3---}. \end{theorem}

Note that for any (irrationally invariant) symmetric $\alpha_0$-stable L\'evy process with $\alpha_0\in(0,2)$, \eqref{e:lll} holds with $\theta\in (0,\alpha_0-\varepsilon)$ for any $\varepsilon>0$ and \eqref{K-} is satisfied with $\theta_0\in (0,\alpha_0]$. In particular, for any symmetric $(1+\varepsilon)$-stable L\'evy process with $\varepsilon\in(0,1)$, both \eqref{e:lll} and \eqref{K-} hold with $\theta=1$. Hence, according to Theorem \ref{T:main2} and the fact that $$|x-x'|+|v-v'|\le (1\wedge (|x-x'|+|v-v'|))\big(1+ |x|+|v|+|x'|+|v'|\big),\quad x,x',v,v'\in \R^d,$$
we can immediately deduce the following statement from Theorem \ref{T:main2}.

\begin{corollary}\label{C:main} Let $(X_t,V_t)_{t\ge0}$ be the process defined by \eqref{E1} such that Assumptions ${\bf(B_0)}$ and ${\bf(B_1)}$ are satisfied, and
 $(L_t)_{t\ge0}$ is a symmetric $\alpha_0$-stable L\'evy process with $\alpha_0\in (0,2)$.
 Then, the assertion of Theorem $\ref{T:main2}$ holds with the standard $L^1$-Wasserstein distance with the metric
 $$((x,v),(x',v'))\mapsto (|x-x'|+|v-v'|)^{1\wedge(\alpha_0-\varepsilon)},\quad x,x',v,v'\in \R^d$$ for any $\varepsilon>0.$  \end{corollary}

\subsection{Approach} We shall make some comments on the approach adopted in  this paper.

First, under the Lyapunov drift condition in Assumption ${\bf (A_1)}$, a standard way to yield the exponential ergodicity of the process $(X_t,V_t)_{t\ge0}$ is to verify that it has the strong Feller property and the irreducible property; see \cite{MT}. But so far it is unclear  whether either of those two properties holds true for the process $(X_t,V_t)_{t\ge0}$ under Assumptions ${\bf(A_0)}$, ${\bf (A_1)}$ and ${\bf(A_2)}$, even for the special  process $(X_t,V_t)_{t\ge0}$ given by \eqref{E1} under ${\bf(B_0)}$, ${\bf(B_1)}$ and ${\bf(B_2)}$. See  \cite{MSH,Talay, Wu} for the ergodicity of kinetic Langevin diffusions via the method based on a Lyapunov drift condition.

The argument based on functional inequalities is one of powerful tools in the study of long time behavior of kinetic Langevin diffusions; see \cite{Vi} and references therein for more details. This methodology relies heavily on the explicit formulation of the associated invariant probability measure, which now is unavailable for the process $(X_t,V_t)_{t\ge0}$ defined by \eqref{EE1} or \eqref{E1} (even when $(L_t)_{t\ge0}$ is a symmetric stable-L\'evy process).

The approach for Theorem \ref{T:main1} is based on the probabilistic coupling method. Markov couplings have been successfully  used to establish the exponential ergodicity of non-degenerate SDEs with L\'evy noises or non-degenerate McKean-Vlasov SDEs with additive L\'{e}vy noises; see \cite{LMW, LWb, LW,Ma}. However, the result for degenerate
SDEs with jumps in this direction is still open. To the best of our knowledge, this is the first paper to investigate via Markov coupling the exponential ergodicity for the  stochastic Hamiltonian systems with L\'evy noises.

Coupling argument has been exploited  to derive
the ergodicity of kinetic Langevin diffusions; see  \cite{BGM, EGZ}.
The remarkable work \cite{EGZ} shows that  the coupling method can not only provide qualitative convergence rate (rather than the quantitative one) to the equilibrium but also a good probabilistic understanding of the dynamics involved. The coupling idea of our paper is partly inspired by that in \cite{EGZ}; that is, instead of considering directly the coupling for the process $(X_t, V_t)_{t\ge0}$, we will study the coupling for the transformed process $(X_t, X_t+\aa^{-1}V_t)_{t\ge0}$ with some proper $\alpha>0$. One of the key observations is that with this transformation the first component can be contractive for the process $(X_t, X_t+\aa^{-1}V_t)_{t\ge0}$ when the distance between two marginal processes of the associated coupling processes is bounded and the impact of the L\'evy noise is small. Despite the similar spirit in the approach, there are still a few of essential differences between kinetic Langevin diffusions and the stochastic Hamiltonian systems with L\'evy noises, which require some new ideas as indicated in the present framework. For example,

\begin{itemize}
\item[(i)] The rate of convergence to equilibrium for kinetic Langevin diffusions was investigated in \cite{EGZ} via a combination of a reflection coupling and a synchronous coupling, which depends on some nice properties of Brownian motion (e.g.\ the L\'evy characterization). Since the L\'evy process $(L_t)_{t\ge0}$ in \eqref{EE1} may be non-symmetric or has a  degenerate L\'evy measure, we will adopt the refined basic coupling for pure jump L\'{e}vy processes, which was initiated in the work  \cite{LW}. We emphasize that the refined basic coupling is more powerful in considering  degenerate SDEs with jumps in the sense that we do not need the approximation technique as taken in \cite{EGZ} for kinetic Langevin diffusions.

\item[(ii)] The other crucial ingredient to prove the exponential ergodicity of the process $(X_t, V_t)_{t\ge0}$ under  the Wasserstein-type distance is to construct a proper cost function by using  some distance-like function. Similar to \cite{EGZ}, we will employ  the multiplicative distance $W_\Psi$ (which is first introduced in \cite{HMS} to establish the weak Harris' theorem). In the present setting, we are concerned with L\'evy noises, then the infinitesimal generator (i.e., coupling operator) of the Markov coupling process is a non-local operator, which makes related estimates for the coupling operator acting on the cost function much more involved. In particular, we need to handle the folded terms   and make sure that such terms can be ignored (which is indeed guaranteed by Assumption ${\bf(A_2)}$(ii)). On the other hand, because   the L\'evy measure may have infinite second (even first) moment, one shall   apply merely the basic refined coupling to the component of L\'evy process $(L_t)_{t\ge0}$ instead of the original L\'evy process $(L_t)_{t\ge0}$ (see Section \ref{section2} for more details). Moreover,  the construction of a Lyapunov condition herein is much more delicate (see the proof of Theorem \ref{T:main2}). Besides, our proof essentially makes full use of the coupling operator only, rather than by means of It\^{o} formula. Note that the latter tool (for example, the formula of the It\^{o} product rule for stochastic integral with jumps) would look tedious in our setting.  Last but not least, with contrast to the existing literature, we can deal with a large class of double well potentials $U_0$ with super-linear growth at infinity in \eqref{E1}.
\end{itemize}

\ \

The rest of this paper is arranged as follows. In Section \ref{section2}, we construct a new Markov coupling for the stochastic Hamiltonian system with L\'evy noises $(X_t,V_t)_{t\ge0}$ defined by \eqref{EE1}, which is partly motivated by the refined basic coupling for  pure jump L\'{e}vy processes introduced in \cite{LW}. In Section \ref{section3}, we present some estimates for the coupling operator associated with the Markov coupling process given in Section \ref{section2}, which are crucial to yield the exponential ergodicity for the process $(X_t,V_t)_{t\ge0}$. In Section \ref{section4}, we present the proofs of Theorems \ref{T:main1} and \ref{T:main2}. In particular, some explicit sufficient conditions put
directly on the coefficients of \eqref{E1aa} and the L\'evy measure $\nu$ are also given here to show  that Assumptions ${\bf(A_1)}$  and ${\bf(A_2)}$ are satisfied.

 \section{A new Markov coupling for stochastic Hamiltonian system with L\'evy noises}\label{section2}
 In this section, we will construct a Markov coupling process for the stochastic Hamiltonian system with L\'evy noises $(X_t,V_t)_{t\ge0}$ solved  by \eqref{EE1}. For this, we will consider a coupling operator for the operator $\mathscr L $ defined by \eqref{EE-}, where the refined basic coupling for pure jump L\'{e}vy processes is fully used. The refined basic coupling for pure jump L\'{e}vy processes was initiated in the work  \cite{LW},  and was further developed   to investigate gradient estimates for SDEs driven by multiplicative L\'{e}vy noises in \cite{LWb} and exponential ergodicity for McKean-Vlasov SDEs with additive L\'{e}vy noises in \cite{LMW}.

Let us first introduce some necessary notation. Given the threshold $\kk>0,$ define $$ (x)_\kk=\left(1\wedge \ff{\kk }{|x|}\right)x  ~\mbox{ for }~{\bf 0}\neq x\in\R^d;~~ (x)_\kk={\bf 0} ~ \mbox{ for }~x={\bf 0}.$$
 For $x\in\R^d$, let $\dd_x$ be the Dirac measure or the unit mass at the point $x$. Let $\nu^*$ be a non-negative measure on $(\R^d, \mathscr{B}(\R^d))$ such that $\nu^*\le \nu$. For $x\in\R^d$,
let
\begin{equation}\label{--}
\nu^*_x(\d u)=\big(\nu^*\wedge(\dd_x*\nu^*)\big)\,(\d u),
\end{equation}
which  indeed is  a finite measure on $(\R^d, \mathscr{B}(\R^d))$ when $x\neq {\bf 0}$. In fact, for  $x\neq {\bf 0}$,
a direct calculation shows
\begin{align*}
\nu^*_x(\R^d)&\le \int_{ \{|u|\ge {|x|}/{2} \}}\,\nu^*(\d u)+\int_{ \{|u|\le {|x|}/{2} \}} \,\nu^* (\d (u-x))\\
&\le 2\int_{ \{|u|\ge {|x|}/{2} \}}\,\nu^*(\d u)\le \ff{8}{|x|^2\wedge 4}\int_{ \{1\wedge|u|\ge 1\wedge{|x|}/{2}  \}}( 1\wedge |u|^2)\,\nu^*(\d u)\\
&\le 8\int_{\R^d}(1\wedge|u|^2)\,\nu^*(\d u)(1\wedge |x|)^{-2}\le  8\int_{\R^d}(1\wedge|u|^2)\,\nu(\d u)(1\wedge |x|)^{-2}<\8,
\end{align*}
where in the second inequality we used the fact that $|u-x|\ge |x|/2$ if  $|u|\le  |x|/2$. See \cite[Appendix]{LW} for detailed properties of the measure $\nu^*_x$. In particular, it is obvious that $\nu^*_x$ is absolutely continuous with respect to $\nu$.

As we will see later, instead of considering the process $(X_t, V_t)_{t\ge0}$ directly, we will study the coupling for $(X_t, X_t+\aa^{-1}V_t)_{t\ge0}$, where $\alpha>0$ is to be determined later.
Then, according to the strategy of \cite{LW} (in particular, see \cite[(2.7)]{LW} therein for more details), the refined basic coupling of the infinitesimal  operator $\mathscr L_0 $ of the pure jump L\'evy process $(L_t)_{t\ge0}$ is constructed via the following relationship, for $x,x',v,v'\in\R^d$ with $q:=x-x'+\aa^{-1}(v-v')$
\begin{align*}
(v,v')\rightarrow
\begin{cases}
 (v+u, v'+u+\aa(q)_\kk),&\quad \ff{1}{2}\nu^*_{-\aa(q)_\kk}(\d u),\\
 (v+u, v'+u-\aa(q)_\kk),&\quad  \ff{1}{2}\nu^*_{\aa(q)_\kk}(\d u),\\
 (v+u,v'+u),&\quad  \big(\nu-\ff{1}{2}\nu^*_{-\aa(q)_\kk}-\ff{1}{2}\nu^*_{\aa(q)_\kk}\big)(\d u),
\end{cases}
\end{align*} where $\nu^*_{\aa(q)_\kk}$ (resp.\ $\nu^*_{-\aa(q)_\kk}$) is defined by \eqref{--} with $x=\aa(q)_\kk$ (resp.\ $x=-\aa(q)_\kk$).
 For $x,x',v,v'\in\R^d $  and $f\in C_b^2(\R^{4d})$, define
\begin{equation}\label{EE}\big(\tilde{\mathscr L}f\big) \big((x,v),(x',v')\big)=\big(\tilde{\mathscr L}_1f\big) \big((x,v),(x',v')\big)+ \big(\tilde{\mathscr L}_{x,x'} \big) f\big((x,\cdot),(x',\cdot)\big)(v,v'),\end{equation}
 where \begin{equation}\label{W2-}
  \begin{split}
 &\big(\tilde{\mathscr L}_1f\big) \big((x,v),(x',v')\big)\\
 &:=\big\<ax+bv,\nn_x  f\big((x,v),(x',v')\big)\big\>+\big\<ax'+bv',\nn_{x'}  f\big((x,v),(x',v')\big)\big\>\\
&\quad\,+\big\< U(x,v),\nn_v f\big((x,v),(x',v')\big)\big\>+\big\<U(x',v'),\nn_{v'} f\big((x,v),(x',v')\big)\big\>
\end{split}
 \end{equation} and, for $g\in C_b^2(\R^{2d})$,
 \begin{equation}\label{E-}
 \begin{split}
 (\tilde{\mathscr L}_{x,x'} g)(v,v')
&=\ff{1}{2}\int_{\R^d}\big(g(v+u,v'+u+\aa(q)_\kk)-g(v,v')-\<\nn_v g(v,v'),u\>\I_{\{|u|\le 1\}}\\
&\quad~~~~~~~~~~-\<\nn_{v'}g(v,v'),u+\aa(q)_\kk\>\I_{\{|u+\aa(q)_\kk|\le 1\}}\big)\,\nu^*_{-\aa(q)_\kk}(\d u)\\
&\quad+\ff{1}{2}\int_{\R^d}\big(g(v+u,v'+u-\aa(q)_\kk)-g(v,v')-\<\nn_v g(v,v'),u\>\I_{\{|u|\le 1\}}\\
&\quad~~~~~~~~~~\,\,-\<\nn_{v'} g(v,v'),u-\aa(q)_\kk\>\I_{\{|u-\aa(q)_\kk|\le 1\}}\big)\,\nu^*_{\aa(q)_\kk}(\d u)\\
&\quad+\int_{\R^d}\big(g(v+u,v'+u)-g(v,v')-\<\nn_v g(v,v'),u\>\I_{\{|u|\le 1\}}\\
&\quad~~~~~~~~~~\,\,-\<\nn_{v'} g(v,v'), u \>\I_{\{|u|\le 1\}}\big)\,\Big(\nu-\ff{1}{2}\nu^*_{-\aa(q)_\kk}-\ff{1}{2}\nu^*_{\aa(q)_\kk}\Big)(\d u).
\end{split}
\end{equation}

 \begin{lemma}\label{Lem1} The operator $\tilde{\mathscr L}$ defined by \eqref{EE} is a coupling operator of $\mathscr L$ given by \eqref{EE-}.

\end{lemma}

 \begin{proof} We only need to verify that, for any $x,x'\in\R^d,$ $\tilde{\mathscr L}_{x,x'} $ is a coupling operator of $\mathscr L_0$ given by \eqref{EE0}.
 Recall that $q=x-x'+\aa^{-1}(v-v')$ for   $x,x',v,v'\in\R^d$. Thus, it is sufficient to check that for $g,h\in C_b^2(\R^d)$,
\begin{equation}\label{EE2}
(\tilde{\mathscr L}_{x,x'}f)(v,v')=(\mathscr L_0 g)(v)+(\mathscr L_0h)(v'),
\end{equation} where $f(v,v'):=g(v)+h(v')$.
Indeed,
\begin{align*}
 (\tilde{\mathscr L}_{x,x'} f)(v,v')
&=\mathscr L_0 g(v)\\
 &\quad+\ff{1}{2}\int_{\R^d}\big(h( v'+u+\aa(q)_\kk)-h( v')-\<\nn h(v'),u+\aa(q)_\kk\>\I_{\{|u+\aa(q)_\kk|\le 1\}}\big)\,\nu^*_{-\aa(q)_\kk}(\d u)\\
&\quad+\ff{1}{2}\int_{\R^d}\big(h( v'+u-\aa(q)_\kk)-h(v')-\<\nn h(v'),u-\aa(q)_\kk\>\I_{\{|u-\aa(q)_\kk|\le 1\}}\big)\,\nu^*_{\aa (q)_\kk}(\d u)\\
&\quad+\int_{\R^d}\big(h(v'+u)-h(v') -\<\nn h(v'), u \>\I_{\{|u|\le 1\}}\big)\,\Big(\nu-\ff{1}{2}\nu^*_{-\aa(q)_\kk}-\ff{1}{2}\nu^*_{\aa(q)_\kk}\Big)(\d u).
\end{align*}
Then,  by changing the variables $u+\aa(q)_\kk\rightarrow u$ and $u-\aa(q)_\kk\to u$, respectively, and by using the facts that (see \cite[(2.3) or Corollary (A2) in the Appendix]{LW} for more details)
\begin{equation}\label{P1}
\nu^*_{-\aa(q)_k}(\d (u-\aa(q)_k))=\nu^*_{\aa(q)_k}(\d u),\quad\nu^*_{\aa(q)_k}(\d (u+\aa(q)_k))=\nu^*_{-\aa(q)_k}(\d u),
\end{equation}
 the desired assertion \eqref{EE2} is available.
 \end{proof}

In the sequel, we shall construct explicitly the coupling process associated with the coupling operator $\tilde{\mathscr L}$ defined by \eqref{EE}.  In terms of the L\'{e}vy-It\^o decomposition (see e.g. \cite[Theorem 2.4.16, p.\,126]{AP}), $(L_t)_{t\ge0}$ can be represented as below
$$
L_t=\int_0^t\int_{\{|u|\le 1\}}u\,\bar N(\d s,\d u)+\int_0^t\int_{\{|u|> 1\}}u\, N(\d s,\d u) ,\quad t\ge0,
$$
where $N(\d t,\d u)$ is  the Poisson random measure  with the intensity measure $\d t\,\nu(\d u)$ and $\bar N(\d t,\d u)$ is its compensated Poisson random measure, i.e.,
$$\bar N(\d t,\d u)= N(\d t, \d u)-\d t \,\nu(\d u).$$ To characterize  the coupling process,  inspired by  \cite[Section 2.2]{Ma} (see also \cite{LWb,LW} for more details), we need to extend the Poisson measure $N$  on $\R_+\times\R^d$  to the counterpart   on $\R_+\times\R^d\times [0,1]$.  Let $(p_t)_{t\ge0}$ be the Poisson point process related to  $(L_t)_{t\ge0}$, i.e., $$p_t=L_t-L_{t-},\quad t\in \mathcal D_p:=\big\{s>0: L_s\neq L_{s-}\big\}.$$
It holds that
$$
N((0,t],U)=\#\big\{s\in \mathcal D_p: s\le t, p_s\in U\big\},\quad t>0, U\in\mathscr B(\R^d),
$$
where $\#\{\cdot\}$ denotes the counting measure.
Let $(p_t^e)_{t\ge0}$ be the extension of the Poisson point process $(p_t)_{t\ge0}$, and $N^e$ the Poisson random measure on $\R_+\times\R^d\times[0,1]$ corresponding to $(p_t^e)_{t\ge0}$, i.e.,
$$
N^e((0,t]\times U)=\#\big\{s\in \mathcal D_p: s\le t, p_s^e\in U\big\},\quad t>0, U\in\mathscr B(\R^d\times[0,1]).
$$
In accordance with  \cite[Chapetr II, Lemma 7.2]{IW}, we infer
$$
\bar N^e(\d t,\d u,\d l)=N^e(\d t,\d u,\d l)-\d t\,\nu(\d u)\,\I_{[0,1]}(l)\,\d l.
$$
In particular, $(L_t)_{t\ge0}$ can be reformulated as
\begin{align*}L_t=&\int_0^t\int_{\{|u|\le 1\}\times[0,1]}u\,\bar N^e(\d s, \d u,\d l)+\int_0^t\int_{\{|u|> 1\}\times[0,1]}u\,N^e(\d s, \d u,\d l)\\
=&:\int_0^t\int_{\R^d\times[0,1]}u\,\tilde N^e(\d s, \d u,\d l),\quad t\ge0.\end{align*}

Now we  consider the following SDE
\begin{equation} \label{E0}
\begin{cases}
\d X_t=(aX_t+bV_t)\,\d t,\\
\d V_t=U(X_t,V_t)\,\d t+ \d L_t,\\
\d X_t'=(aX_t'+bV_t')\,\d t,\\
\d V_t'=U(X_t', V_t')\,\d t+ \d L_t^*.
\end{cases}
\end{equation}
Herein,
\begin{align*}
  L_t^*:=\int_0^t\int_{\R^d\times[0,1]}\Big\{&((u+\aa(Q_{s-})_\kk)\I_{\{l\le\ff{1}{2}\rho(-\aa(Q_{s-})_\kk,u)\}}\\
&+(u-\aa(Q_{s-})_\kk)\I_{\{\ff{1}{2}\rho(-\aa(Q_{s-})_\kk,u)<l\le \ff{1}{2}(\rho(-\aa(Q_{s-})_\kk,u)+\rho(\aa(Q_{s-})_\kk,u))\}}\\
&+u\I_{\{\ff{1}{2}(\rho(-\aa(Q_{s-})_\kk,u)+\rho(\aa(Q_{s-})_\kk,u))<l\le 1\}}\Big\}\,\tilde N^e(\d t, \d u,\d l),
\end{align*}
where
\begin{equation} \label{B12}
Q_t:=Z_t+\aa^{-1}W_t ~~\mbox{ with }~~Z_t:=X_t-X_t' ~\mbox{ and }~ W_t:=V_t-V_t'
\end{equation} and \begin{equation}\label{EE3}
\rho(x,u):=\frac{\nu_x^*(\d u)}{\nu(\d u)},\quad x,u\in \R^d.
\end{equation}
A straightforward calculation shows
\begin{equation} \label{E5}
\d L_t^*=\d L_t+\aa(Q_{t-})_\kk\int_{\R^d\times[0,1]}\Lambda(\aa(Q_{t-})_\kk,u,l)\,\tilde N^e(\d t, \d u,\d l),
\end{equation}
where, for $x,u\in\R^d$ and $l\in[0,1]$,
$$\Lambda(x,u,l):=\I_{\{l\le\ff{1}{2}\rho(-x,u)\}}-\I_{\{\ff{1}{2}\rho(-x,u)<l\le \ff{1}{2}(\rho(-x,u)+\rho(x,u))\}}.$$
By invoking \eqref{E5}, \eqref{E0} can be rewritten as
\begin{equation}\label{EE4}
\begin{cases}
\d X_t=(aX_t+bV_t)\,\d t,\\
\d V_t\,=U(X_t,V_t)\,\d t+ \d L_t,\\
\d X_t'=(aX'_t+bV'_t)\,\d t,\\
\d V_t'=U(X_t', V_t')\,\d t+ \d L_t \\
\qquad\quad+\aa(Q_{t-})_\kk\displaystyle \int_{\R^d\times[0,1]}\Lambda(\aa(Q_{t-})_\kk,u,l)\,\tilde N^e(\d t, \d u,\d l).
\end{cases}
\end{equation}

Note that, under Assumptions ${\bf (A_0)}$ 
and ${\bf (A_1)}$ the SDE \eqref{EE1} has a unique strong solution $(X_t, V_t)_{t\ge0}$. Then, by following the proof of \cite[Proposition 2.2]{LW}, \eqref{E0} also has a unique strong solution $((X_t,V_t),(X_t',V_t'))_{t\ge0}$. Furthermore,
 the following statement indicates that $((X_t,V_t),(X_t',V_t'))_{t\ge0}$ solving \eqref{E0} (i.e.\ \eqref{EE4}) is indeed  a coupling process corresponding to $(X_t,V_t)_{t\ge0}$ satisfying \eqref{EE1}.
 \begin{proposition}
The infinitesimal generator of the process $((X_t,V_t),(X_t',V_t'))_{t\ge0}$ is just the operator $\tilde{\mathscr L}$ given by \eqref{EE}. Consequently, $((X_t,V_t),(X_t',V_t'))_{t\ge0}$ is a coupling process of $(X_t,V_t)_{t\ge0}$.
 \end{proposition}

 \begin{proof}
 Let $\hat{\mathscr L}$ be the infinitesimal generator of $((X_t,V_t),(X_t',V_t'))_{t\ge0}$. Recall that $q =x-x'+\aa^{-1}(v-v')$ for any $x,x',v,v'\in\R^d$.
 Then, according to the structure of \eqref{E0},  we derive from \eqref{EE3} that for any $f\in C_b^2(\R^{4d})$,
 \begin{align*}
(\hat{\mathscr L}f) ((x,v),(x'v'))&=(\tilde{\mathscr L}_1f) ((x,v),(x',v'))\\
&\quad+\int_{\R^d\times[0,1]}\Big(f\big((x,v+  u),(x',v'+  (u+\aa(q)_\kk)\I_{\{l\le\ff{1}{2}\rho(-\aa(q)_\kk,u)\}}\\
&\quad\qquad+(u-\aa(q)_\kk)\I_{\{\ff{1}{2}\rho(-\aa(q)_\kk,u)<l\le \ff{1}{2}(\rho(-\aa(q)_\kk,u)+\rho(\aa(q)_\kk,u))\}}\\
&\quad\qquad +u\I_{\{\ff{1}{2}(\rho(-\aa(q)_\kk,u)+\rho(\aa(q)_\kk,u))<l\le 1\}})\big)\\
&\quad-f((x,v),(x'v'))-\<\nn_v f((x,v),(x',v')),  u\>\I_{\{|u|\le 1\}}\\
&\quad-\<\nn_{v'} f((x,v),(x',v')),  (u+\aa(q)_\kk)\I_{\{|u+\aa(q)_\kk|\le1, l\le\ff{1}{2}\rho(-\aa(q)_\kk,u)\}}\\
&\quad\qquad+(u+(-q)_\kk)\I_{\{|u+(-q)_\kk|\le 1,\ff{1}{2}\rho((-w)_\kk,u)<l\le \ff{1}{2}(\rho(-\aa(q)_\kk,u)+\rho(\aa(q)_\kk,u))\}}\\
&\quad\qquad+u\I_{\{|u|\le 1,\ff{1}{2}(\rho(-\aa(q)_\kk,u)+\rho(\aa(q)_\kk,u))<l\le 1\}}\>\Big)\,\nu(\d u)\,\d l\\
&=(\tilde{\mathscr L}_1f) ((x,v),(x',v'))+(\tilde{\mathscr L}_{x,x'}f) ((x,\cdot),(x',\cdot))) ((v,v'))\\
&=(\tilde{\mathscr L} f)((x,v),(x',v')),
\end{align*}
 which proves the desired assertion.
 \end{proof}

 \section{Construction of cost function and related estimates}\label{section3}
This section is devoted to the construction of the cost function $\widetilde \Psi$ in the Wasserstein distance (which is comparable with $\Psi$ in Theorem \ref{T:main1}) and to presenting some related estimates, both of which are crucial to obtain exponential ergodicity of the process $(X_t,V_t)_{t\ge0}$.

 \subsection{Rough estimates}
 For any $x,x',v,v'\in\R^d$ and $\aa_0,\vv>0$,  we set
\begin{equation}\label{W1-}
\begin{split}
&z:=x-x',\quad w:=v-v',\quad q:=z+\aa^{-1}w,\quad r:=\aa_0|z|+|q|,\\
& H\big((x,v),(x',v')\big):=f(r),\quad
 G\big((x,v),(x,v')\big):=1+\vv(\mathcal W(x,v)+\mathcal W(x',v')).
 \end{split}
 \end{equation} 
Here,  $f$ is chosen to satisfy that  $f(0)=0$, $f'>0$ and $f''<0$ on $(0,\infty)$ whose explicit expression will be given in the next part, and $\mathcal W$ is the Lyapunov function given in Assumption {$\bf(A_1)$}.

 \begin{lemma}\label{Lem}
 For the coupling operator $\tilde{\mathscr L}$ given in \eqref{EE}, it holds that
 \begin{equation}\label{B4}
\big(\tilde{\mathscr L}(H G) \big) \big((x,v),(x',v')\big)=\big( H(\tilde{\mathscr L}  G)+G(\tilde{\mathscr L} H) +\Pi\big)\big((x,v),(x',v')\big),
\end{equation}
 where
 \begin{equation}\label{B2}
 \begin{split}
\Pi&\big((x,v),(x',v')\big)\\&:=\ff{1}{2}\int_{\R^d}\Big(H\big((x,v+  u),(x',v'+  u+\aa(q)_\kk  )\big) -H\big((x,v),(x',v')\big)\Big)\\
&\qquad\times\Big(G \big((x,v+  u),(x',v'+  u+\aa(q)_\kk  )\big) - G \big((x,v),(x',v')\big)\Big)\,\nu^*_{-\aa(q)_\kk}(\d u)\\
&\quad+\ff{1}{2}\int_{\R^d}\Big(H\big((x,v+  u),(x',v'+  u-\aa(q)_\kk  )\big) - H\big((x,v),(x',v')\big)\Big)\\
&\qquad\times\Big(G\big((x,v+  u),(x',v'+  u-\aa(q)_\kk ) \big) - G\big((x,v),(x',v')\big)\Big)\,\nu^*_{\aa(q)_\kk}(\d u).
\end{split}
\end{equation}
\end{lemma}
\begin{proof}
Recall that $\tilde{\mathscr L}=\tilde{\mathscr L}_1+\tilde{\mathscr L}_{x,x'}$, where  $\tilde{\mathscr L}_1$ and  $\tilde{\mathscr L}_{x,x'}$ are defined by \eqref{W2-} and \eqref{E-}, respectively.
By the chain rule, it is easy to see that
\begin{equation}\label{P2}
\big(\tilde{\mathscr L_1} (H  G)\big) \big((x,v),(x',v')\big) = \big(G\big(\tilde{\mathscr L_1}H\big )+H\big(\tilde{\mathscr L_1}G\big)\big) \big((x,v),(x',v')\big).
\end{equation}
On the other hand, \begin{align*}
&\big(\tilde{\mathscr L}_{x,x'}  (H G)\big) \big((x,\cdot),(x',\cdot)\big)(v,v')\\
&=\ff{1}{2}\int_{\R^d}\Big((H G)\big((x,v+  u),(x',v'+  u+\aa(q)_\kk)\big)-(H G)\big((x,v),(x',v')\big)\\
&\quad~~~~~~~~~~-\<\nn_v(H G)\big((x,v),(x',v')\big),  u\>\I_{\{|u|\le 1\}}\\
&\quad~~~~~~~~~~ -\<\nn_{v'} (H G)\big((x,v),(x',v')\big),  u+\aa(q)_\kk \>\I_{\{|u+\aa(q)_\kk |\le 1\}}\Big)\,\nu^*_{-\aa(q)_\kk}(\d u)\\
&\quad+\ff{1}{2}\int_{\R^d}\Big((H G)\big((x,v+  u),(x',v'+  u-\aa(q)_\kk)\big)-(H G)\big((x,v),(x',v')\big)\\
&\quad~~~~~~~~~~~-~\<\nn_v(H G)\big((x,v),(x',v')\big),  u\>\I_{\{|u|\le1\}}\\
&\quad~~~~~~~~~~~~ -\<\nn_{v'}(HG)\big((x,v),(x',v')\big),   u-\aa(q)_\kk \>\I_{\{| u-\aa(q)_\kk |\le 1\}}\Big)\,\nu^*_{\aa(q)_\kk}(\d u)\\
&\quad+\int_{\R^d}\Big((H G)\big((x,v+  u),(x',v'+ u)\big)-(H G)\big((x,v),(x',v')\big)\\
&\quad~~~~~~~~~~~- \<\nn_v(H G)\big((x,v),(x',v')\big),u\>\I_{\{|u|\le 1\}}- \<\nn_{v'}(H G)\big((x,v),(x',v')\big),u \>\I_{\{|u|\le 1\}}\Big)\\
&\quad~~~~\quad\quad\times\Big(\nu-\ff{1}{2}\nu^*_{-\aa(q)_\kk}-\ff{1}{2}\nu^*_{\aa(q)_\kk}\Big)(\d u)\\
&=:I_1+I_2+I_3.
\end{align*}
According to the chain rule again, it follows that
\begin{align*}
I_1&=\ff{1}{2} G \big((x,v),(x',v')\big)\int_{\R^d}\Big( H\big((x,v+  u),(x',v'+  u+\aa(q)_\kk  )\big)\\
& \quad\qquad\qquad- H \big((x,v),(x',v')\big)  -\<\nn_v H \big ((x,v),(x',v')\big),  u\>\I_{\{|u|\le 1\}}\\
     &\quad\qquad\qquad -\<\nn_{v'} H \big((x,v),(x',v')\big),   u+\aa(q)_\kk \>\I_{\{| u+\aa(q)_\kk |\le 1\}}\Big)\,\nu^*_{-\aa(q)_\kk}(\d u) \\
&\quad+\ff{1}{2} H\big((x,v),(x',v')\big)\int_{\R^d}\Big( G\big((x,v+  u),(x',v'+  u+\aa(q)_\kk  )\big) \\
&\quad \qquad\qquad- G  \big((x,v),(x',v')\big) -\<\nn_v G \big((x,v),(x',v')\big), u\>\I_{\{|u|\le 1\}} \\
  &\quad\qquad\qquad-\<\nn_{v'}  G\big((x,v),(x',v')\big),   u+\aa(q)_\kk \> \I_{\{| u+\aa(q)_\kk|\le 1\}}\Big)\nu^*_{-\aa(q)_\kk}(\d u)\\
&\quad+\ff{1}{2}\int_{\R^d}\Big(H\big((x,v+  u),(x',v'+  u+\aa(q)_\kk  )\big) -H\big((x,v),(x',v')\big)\Big)\\
&\quad\qquad\qquad\times\Big(G \big((x,v+  u),(x',v'+  u+\aa(q)_\kk  )\big) - G \big((x,v),(x',v')\big)\Big)\nu^*_{-\aa(q)_\kk}(\d u) \\
\end{align*}
and that
\begin{align*}
I_2&=\ff{1}{2} G\big((x,v),(x',v')\big)\int_{\R^d}\Big( H\big((x,v+  u),(x',v'+  u-\aa(q)_\kk  )\big) \\
&\qquad\quad\qquad- H \big ((x,v),(x',v')\big)  -\<\nn_vH\big((x,v),(x',v')\big),  u\>\I_{\{|u|\le 1\}}\\
     &\quad\qquad\qquad-\<\nn_{v'} H \big((x,v),(x',v')\big),   u-\aa(q)_\kk \>\I_{\{| u-\aa(q)_\kk|\le 1\}}\Big)\nu^*_{\aa(q)_\kk}(\d u) \\
&\quad+\ff{1}{2} H\big((x,v),(x',v')\big)\int_{\R^d}\Big( G\big((x,v+  u),(x',v'+  u-\aa(q)_\kk  )\big)  \\
&\quad\qquad\qquad - G  \big((x,v),(x',v')\big)-\<\nn_v G  \big((x,v),(x',v')\big), u\> \I_{\|u|\le 1\}}\\
  &\quad\qquad\qquad -\<\nn_{v'}  G\big((x,v),(x',v')\big),   u-\aa(q)_\kk \> \I_{\{|u-\aa(q)_\kk|\le 1\}}\Big)\nu^*_{\aa(q)_\kk}(\d u)\\
&\quad+\ff{1}{2}\int_{\R^d}\Big(H\big((x,v+  u),(x',v'+  u-\aa(q)_\kk  )\big) - H\big((x,v),(x',v')\big)\Big)\\
&\quad\qquad\qquad\times\Big(G\big((x,v+  u),(x',v'+  u-\aa(q)_\kk ) \big) - G\big((x,v),(x',v')\big)\Big)\nu^*_{\aa(q)_\kk}(\d u).
\end{align*}
Moreover, due to the structure of the function $H$,
\begin{align*}
&\int_{\R^d}\Big(H\big((x,v+  u),(x',v'+  u ) \big)- H\big((x,v),(x',v')\big)\Big)\\
 &\quad\quad\quad\times\Big(G \big((x,v+  u),(x',v'+  u)\big) - G \big((x,v),(x',v')\big)\Big)\\
&\quad\times\Big(\nu-\ff{1}{2}\nu_{-\aa(q)_\kk}-\ff{1}{2}\nu^*_{\aa(q)_\kk}\Big)(\d u)=0
\end{align*}
and so we have
\begin{align*}
I_3&= G\big((x,v),(x',v')\big)\int_{\R^d}\Big( H\big((x,v+  u),(x',v'+  u )\big)- H  \big((x,v),(x',v')\big)  \\
&\quad\qquad\qquad -\<\nn_v H \big ((x,v),(x',v')\big),  u\>\I_{\{|u|\le 1\}}     -\<\nn_{v'} H\big ((x,v),(x',v')\big),   u \>\I_{\{|u|\le 1\}}\Big)\\
&\quad\qquad\qquad\times\Big(\nu-\ff{1}{2}\nu^*_{-\aa(q)_\kk}-\ff{1}{2}\nu^*_{\aa(q)_\kk}\Big)(\d u)\\
&\quad+ H\big((x,v),(x',v')\big)\int_{\R^d}\Big( G\big((x,v+  u),(x',v'+  u )\big )- G \big ((x,v),(x',v')\big)  \\
&\quad\qquad\qquad -\<\nn_v G  \big((x,v),(x',v')\big), u\> \I_{\{|u|\le 1\}}  -\<\nn_{v'}  G\big((x,v),(x',v')\big),   u\>\I_{\{|u|\le 1\}} \Big)\\
&\quad\qquad\qquad \times\Big(\nu-\ff{1}{2}\nu^*_{-\aa(q)_\kk}-\ff{1}{2}\nu^*_{\aa(q)_\kk}\Big)(\d u).
\end{align*}

Combining all identities above, we derive
\begin{equation}\label{P3}
\begin{split}
&\big(\tilde{\mathscr L}_{x,x'}  (H G)\big) \big((x,\cdot),(x',\cdot)\big)(v,v')\\
&=G\big((x,v),(x',v')\big)\big(\tilde{\mathscr L}_{x,x'}  H\big) \big((x,\cdot),(x',\cdot)\big)(v,v')\\
&\quad + H\big((x,v),(x',v')\big)\big(\tilde{\mathscr L}_{x,x'}  G\big) \big((x,\cdot),(x',\cdot)\big)(v,v')+ \Pi\big((x,v),(x',v')\big).
\end{split}
\end{equation}

Consequently, \eqref{B4} follows from
\eqref{P2} and \eqref{P3} immediately.
\end{proof}

For our further use, we need the following two lemmas.

\begin{lemma}\label{Lemma1}For the coupling operator $\tilde{\mathscr L}$  given in \eqref{EE}, it holds that \begin{equation}\label{e:esi-drift}\begin{split}
\big(\tilde{\mathscr L} H\big) \big((x,v),(x',v')\big)&= f'(r)\Big\{\aa_0(a-b\aa)|z|+\ff{b\aa\aa_0}{|z|}\< z,q\>\\
&\qquad \qquad +\ff{1}{|q|}\big\<q,az+bw+\aa^{-1}(U(x,v)-U(x',v'))\big\>\Big\}\\
&\quad+\ff{1}{2}\big(f(r+\kk\wedge |q|)+f(r-(\kk\wedge|q|))-2f(r)\big)\nu_{\aa(q)_\kk}^*(\R^d).
\end{split}\end{equation}
\end{lemma}

\begin{proof} We still let $\tilde{\mathscr L}_1$ and  $\tilde{\mathscr L}_{x,x'}$ be given by \eqref{W2-} and \eqref{E-}, respectively.
Firstly, acting $\tilde{\mathscr L}_1$ on  $H$  yields
\begin{align*}
&\big(\tilde{\mathscr L}_1H\big) \big((x,v),(x',v')\big)\\
&=f'(r)\Big\{\ff{ \aa_0}{|z|} \<z, az+bw \>
 +\ff{1}{|q|}\big\<q,az+bw+\aa^{-1}(U(x,v)-U(x',v'))\big\>\Big\}\\
 &= f'(r)\Big\{\aa_0(a-b\aa)|z|+\ff{b\aa\aa_0}{|z|}\< z,q\>
 +\ff{1}{|q|}\big\<q,az+bw+\aa^{-1}(U(x,v)-U(x',v'))\big\>\Big\},
\end{align*}
where in the second identity we used $q=z+\aa^{-1}w.$

Secondly,
\begin{align*}
 \big(\mathscr L_{x,x'}H((x,\cdot), (x',\cdot))\big) (v,v')
&=\ff{1}{2}\int_{\R^d}\Big(f(\aa_0|z|+|q- (q)_\kk |)-f(r)- \ff{f'(r)}{\aa|q|}\<q,  u \>\I_{\{|u|\le 1\}}\\
&\qquad\qquad\quad+\ff{f'(r)}{\aa|q|}\<q,  u+\aa(q)_\kk \>\I_{\{|u+\aa(q)_\kk|\le 1\}}\Big)\,\nu^*_{-\aa(q)_\kk}(\d u)\\
&\quad+\ff{1}{2}\int_{\R^d}\Big(f(\aa_0|z|+|q+ (q)_\kk |)-f(r)- \ff{f'(r)}{\aa|q|}\<q,  u \>\I_{\{|u|\le 1\}}\\
&\qquad\qquad\quad +\ff{f'(r)}{\aa|q|}\<q,  u-\aa(q)_\kk \>\I_{\{|u-\aa(q)_\kk|\le 1\}}\Big)\,\nu^*_{\aa(q)_\kk}(\d u)\\
&=\ff{1}{2}\big(f(\aa_0|z|+|q- (q)_\kk |)+f(\aa_0|z|+|q+ (q)_\kk |)-2f(r)\big)\nu^*_{\aa(q)_\kk}(\R^d)\\
&=\ff{1}{2}\big(f(r-\kk\wedge |q|)+f(r+(\kk\wedge|q|))-2f(r)\big)\nu^*_{\aa(q)_\kk}(\R^d),
\end{align*}
where  in the second identity we have changed the variables $u+\aa(q)_\kk$ and $u-\aa(q)_\kk $ into the variable $u$, respectively, and employed the fact \eqref{P1}.

Combining both identities above yields the desired assertion.
\end{proof}

\begin{remark} \label{R:funcitonf} Since in the argument below (see the proof of Theorem \ref{T:main1}, which is partly referred to that of \cite[Proposition 4.3]{LW}) we essentially apply the 
 (Dynkin) martingale formula of  
 the Markov coupling process  $((X_t,V_t),(X_t',V_t'))_{t\ge0}$ defined by \eqref{E0}, by the expression of  $ \big(\tilde{\mathscr L} H\big) \big((x,v),(x',v')\big)$ given in Lemma \ref{Lemma1}, we know that  $ \big(\tilde{\mathscr L} H\big) \big((x,v),(x',v')\big)$
is well defined for any piecewise $C^1$ function $f$ on $(0,\infty)$ such that the left-sided first derivative $f'_-$ is finite almost everywhere. In this case, $f'$ is replaced by $f'_-$ in the right hand side of \eqref{e:esi-drift}. \end{remark}

\begin{lemma}\label{lem5}
Under $({\bf A_2})${\rm(ii)}, for any $x,x',v,v'\in\R^d,$
$$
\Pi\big((x,v),(x',v')\big)\le 2c_{*}\vv H\big((x,v),(x',v')\big)\big(\mathcal W(x,v)^{\eta}+\mathcal W(x',v')^{\eta}\big),
$$
where $\Pi$ is defined  by \eqref{B2} and $c_*$ is the constant in Assumption ${\bf(A2)}${\rm (ii)}.
\end{lemma}

\begin{proof}
Let $\Pi_1 $ and $\Pi_2 $ be the two terms of the right hand side in the definition of $\Pi$ given in \eqref{B2}. Substituting $u+\aa(q)_\kk$ by $u$ and making use of \eqref{P1}, we have
\begin{align*}
&\Pi_1\big((x,v),(x',v')\big)\\
& =\ff{\vv}{2}\big(f(r-(\kk\wedge |q|)) -f(r)\big)\\
&\quad\times \int_{\R^d}\Big(\mathcal W(x,v+u)- \mathcal W(x,v)+\mathcal W(x',v'+  u+\aa(q)_\kk) -\mathcal W(x',v')\Big)\,\nu^*_{-\aa(q)_\kk}(\d u)\\
&=\ff{\vv}{2}\big(f(r-(\kk\wedge |q|)) -f(r)\big)\bigg\{\int_{\R^d}\Big(\mathcal W(x,v+u)- \mathcal W(x,v)\Big)\,\nu^*_{-\aa(q)_\kk}(\d u)\\
&\qquad\qquad\qquad\qquad\qquad\qquad\quad+\int_{\R^d}\Big(\mathcal W(x',v'+  u) -\mathcal W(x',v')\Big)\,\nu^*_{\aa(q)_\kk}(\d u)\bigg\}\\
&\le \frac{\vv}{2} f(r)\bigg\{ \int_{\R^d}\Big|\mathcal W(x,v+u)- \mathcal W(x,v)\Big|\,\nu^*_{-\aa(q)_\kk}(\d u)\\
&\qquad\qquad +\int_{\R^d}\Big|(\mathcal W(x',v'+  u) -\mathcal W(x',v')\Big|\,\nu^*_{\aa(q)_\kk}(\d u)\bigg\}\\
&\le \frac{\vv}{2} f(r)\bigg\{ \int_{\R^d}\Big|\mathcal W(x,v+u)- \mathcal W(x,v)\Big|\,\nu^* (\d u) +\int_{\R^d}\Big|(\mathcal W(x',v'+  u) -\mathcal W(x',v')\Big|\,\nu^*(\d u)\bigg\}\\
&\le   \frac{c_* \vv}{2} f(r) \big( W(x,v)^{\eta}+  W(x',v')^{\eta}\big),
\end{align*} where we used Assumption ${\bf(A_2)}$(ii).
Following the procedure to derive the above estimate, we find
\begin{align*}
&\Pi_2\big((x,v),(x',v')\big)\\
&=\ff{\vv}{2}\big(f(r+(\kk\wedge |q|)) +f(r)\big)\bigg\{\int_{\R^d}\Big(\mathcal W(x,v+u)- \mathcal W(x,v)\Big)\,\nu^*_{-\aa(q)_\kk}(\d u)\\
&\qquad\qquad\qquad\qquad\qquad\qquad\quad+\int_{\R^d}\Big(\mathcal W(x',v'+  u) -\mathcal W(x',v')\Big)\,\nu^*_{\aa(q)_\kk}(\d u)\bigg\}\\
&\le \frac{3 c_*}{2} \vv f(r) \big( W(x,v)^{\eta}+  W(x',v')^{\eta}\big),
\end{align*} where in the last inequality we used the property that
$$
f(2r)=f(r)+\int_0^{r}f'(s+r)\,\d s\le f(r)+\int_0^{r}f'(s)\,\d s=2f(r),\quad r>0,
$$ thanks to the fact $f(0)=0$ and $f''\le0$.

With the aid of both estimates above, we get the desired assertion.
\end{proof}

\subsection{Precise estimates} Recall from \eqref{W1-} that for any $x,x',v,v'\in\R^d$,
$$z=x-x',\quad w=v-v',\quad q=z+\aa^{-1}w,\quad r=\aa_0|z|+|q|.$$
For $c_0$, $C_0>0$  given in (${\bf A_1}$) and $c_{*}>0$, $\eta\in(0,1)$   in (${\bf A_2}$)(ii),
set
\begin{align*}
R_0:=&\sup\big\{r:  2C_0\!+ \!  2c_*(\mathcal W(x,v)^{\eta}\!+\!\mathcal W(x',v')^{\eta})\ge c_0(\mathcal W(x,v)\!+\!\mathcal W(x',v'))/2\big\}\\
&+(1+\aa)\kappa+1.
\end{align*}
It is easy to see that $R_0$ above is finite due to $\mathcal W(x,v)\to\8$ as $|x|+|v|\to\8$.
We further set
$$\lambda^*(R_0):=\sup\left\{ \frac{|U(x,v)-U(x',v')|}{|x-x'|+|v-v'|} : r<R_0\right\}.$$
By Assumption ${\bf(A_0)}$ and the definition of $R_0$, we know that $\lambda^*(R_0)<\infty$.

 In this subsection, $\aa$, $\aa_0>0$  are specified as follows:
\begin{equation}\label{E14}
\begin{split}
\aa&=1,\quad\quad\quad\,\,\aa_0=1+\frac{16\lambda^*(R_0)}{b}, \quad\quad\quad \quad\quad \mbox{ when } a=0,\\
\aa&=\ff{16a}{b},\quad\quad\aa_0=3+\Big(\ff{1}{a}+\ff{b}{16a^2}\Big)\lambda^*(R_0),\quad\mbox{ when } a\neq0.
\end{split}
\end{equation}

Let $r_0>0$ be the constant and $\si_{r_0}$ be the function given in Assumption ${\bf(A_2)}$.  See $ \kk= r_0/(2\aa)$. Since $\si_{r_0}$ is a non-decreasing function on $[0,r_0],$
 we deduce from Assumption ${\bf(A_2)}$ that
 for all $x\in \R^d$ with $|x|\le R_0$,
\begin{equation}\label{e:ppff}
\begin{split}
\ff{1}{|x|}J(\aa(\kk\wedge |x|))( \kk\wedge |x|)^2
&\ge \aa^{-1}\si_{r_0}(\aa(\kk\wedge |x|)) (1\wedge \kk/ |x|)\\
&\ge \aa^{-1}( 1\wedge\kk/ R_0)\si_{r_0}(\aa (1\wedge\kk/R_0)|x|).
\end{split}
\end{equation}
Let
$$
\si_{\aa,\kappa,R_0}(s)=\aa^{-1}( 1\wedge\kk/ R_0)\si_{r_0}(\aa(1\wedge\kk/R_0)s),\quad s\in[0,2R_0].
$$
Then,  (${\bf A_2}$)(i)  implies that $\si_{\aa,\kappa,R_0}\in C([0,2R_0])\cap C^2((0,2R_0])$
 is a non-decreasing and concave function  such that $\int_{0}^{2R_0} \si_{\aa,\kappa,R_0}(l)\,\d l<\infty$, and, for all $x\in \R^d$ with $|x|\le R_0$,
\begin{equation}\label{W0}
\si_{\aa,\kappa,R_0}(|x|)\le \ff{1}{|x|}J(\aa(\kk\wedge |x|))( \kk\wedge |x|)^2,
\end{equation}thanks to \eqref{e:ppff}.
Furthermore,  define
$$
g(s)=C_*\int_0^s\ff{1}{\si_{\aa,\kappa,R_0}(l/(1+k_0\aa_0))}\,\d l,\quad s\ge0,
$$
where \begin{equation}\label{E15}
\begin{split}
k_0:&=\ff{8\big(\lambda^*(R_0)+b\aa(1+\aa_0)+3(1-1/\aa_0)b\aa/4\big)}{(\aa_0-1)b\aa},\\
\Lambda_0:&= (k_0a+b\aa)(1+\aa_0)+\lambda^*(R_0)(1+(1+1/\aa)k_0),\\
C_*:&=1+\ff{8\Lambda_0}{3(1-1/\aa_0)b\aa}.
\end{split}
\end{equation}

Now, for $c_1:=\e^{-c_2g(2R_0)}$ with $c_2:={3}(1-1/\aa_0)b\aa(1+k_0\aa_0)/2$, we set
\begin{equation} \label{E16}
\begin{split}
&f(s):=c_1s+\int_0^s\e^{-c_2g(l)}\,\d l,\quad s\ge0,\\
& \hat H\big((x,v),(x',v')\big):=f(r\wedge R_0),\quad (x,v),(x',v') \in \R^{2d}.
\end{split}
\end{equation}
By some calculations, we find that for $s\in(0,2R_0]$,
\begin{align*}
g'(s)&=  \frac{C_*}{\si_{\aa,\kk,R_0}(s/(1+k_0\aa_0))}\ge0,\\
g''(s)&=-\ff{C_*\si_{\aa,\kk,R_0}'(s/(1+k_0\aa_0))}{(1+k_0\aa_0)\si_{\aa,\kk,R_0}(s/(1+k_0\aa_0))^{2}}\le0,\\
g^{(3)}(s)&=\ff{2C_*\si_{\aa,\kk,R_0}'(s/(1+k_0\aa_0))^2}{(1+k_0\aa_0)^2\si_{\aa,\kk,R_0}(s/(1+k_0\aa_0))^{3}}-\ff{C_*\si_{\aa,\kk,R_0}''(s/(1+k_0\aa_0))}{(1+k_0\aa_0)^2\si_{\aa,\kk,R_0}(s/(1+k_0\aa_0))^{2}}\ge0.
\end{align*}
So Lemma \ref{lem3} below with $l_0=R_0$ and $c=c_1$, as well as $g$ in place of $c_2g$, is applicable for the function $f$, which will be frequently used later. Note that the function $r\mapsto f(r\wedge R_0)$ is piecewise $C^1$ on $(0,\infty)$ such that $f'_-(r)=0$ for all $r\ge R_0$ and $f'_-(r)-f'(r)$ for all $r<R_0$.
Furthermore, for the Lyapunov function $\mathcal W$ in (${\bf A_1}$),
we define
\begin{equation}\label{E17}
G\big((x,v),(x',v')\big)=1+\vv(\mathcal W(x,v)+\mathcal W(x',v')),
\end{equation}
where
\begin{equation} \label{E18} \vv:=\ff{3c_1(1-1/\aa_0)b\aa}{16(1+c_1)}\Big(2C_0+(1-\eta)\big(2c_*(\eta/c_0)^{\eta}\big)^{{1}/({1-\eta})}\Big)^{-1}.
\end{equation}

 We remark that all constants and functions constructed above are seemingly unusual whereas they will become more and more apparent from the proofs below. In this subsection, we will always fix the functions $f$, $\hat H$ and $G$, as well as the constant $\vv$.

\begin{lemma}\label{Lemma3.4}
 Assume  that $({\bf A_1})$ and $({\bf A_2})${\rm(ii)} hold. Then,
 for all $r\ge R_0$,
$$
\big(\tilde{\mathscr L}(\hat H G) \big) \big((x,v),(x',v')\big)\le  -\ff{c_0\vv}{1+2\vv}(\hat HG) \big((x,v),(x,v')\big).
$$
\end{lemma}

\begin{proof}
Noting that $\tilde{\mathscr L}$ is the coupling operator of $ \mathscr L $ and taking advantage of Assumption ${\bf(A_1)}$, we arrive at
\begin{align*}
\big(\tilde{\mathscr L}\,  G \big) \big((x,v),(x',v')\big)&=\vv\big( (\mathscr L \mathcal W)(x,v)+ (\mathscr L \mathcal W)(x',v')\big)\\
&\le \vv\left(2C_0 -  c_0\big(\mathcal W(x,v)+\mathcal W(x',v')\big)\right).
\end{align*}
This, together  with Lemmas \ref{Lem}, \ref{Lemma1} and \ref{lem5} (as well as Remark \ref{R:funcitonf}), leads to
\begin{equation}\label{B15}
\begin{split}
\big(&\tilde{\mathscr L}(\hat H G) \big) \big((x,v),(x',v')\big)\\
&\le \vv f(r\wedge R_0 )\big( 2C_0-c_0\big(\mathcal W(x,v)+\mathcal W(x',v')\big)\big)\\
&\quad+G\big((x,v),(x',v')\big)\Big[f_-'(r\wedge R_0 )\\
&\quad \quad\times \Big(\aa_0(a-b\aa)|z|+\ff{b\aa\aa_0}{|z|}\< z,q\> +\ff{1}{|q|}\big\<q,az+bw+\aa^{-1}(U(x,v)-U(x',v'))\big\>\Big)\\
&\quad +\ff{1}{2}\big(f(r\wedge R_0 +(\kk\wedge |q|))+f( r\wedge R_0 -(\kk\wedge|q| ))-2f(r \wedge R_0)\big)\nu^*_{\aa(q)_\kk}(\R^d)\Big]\\
&\quad +  2c_*\vv f(r\wedge R_0 )\big(\mathcal W(x,v)^{\eta}+\mathcal W(x',v')^{\eta}\big) \\
&=:\Theta_1+\Theta_2+\Theta_3.
\end{split}\end{equation}

 For $r\ge R_0$, $f(r\wedge R_0)= f(R_0)$ and $f_-'(r\wedge R_0)=0$, so that
$$
\Theta_2=\ff{1}{2}G\big((x,v),(x',v')\big)
\nu^*_{\aa(q)_\kk}(\R^d)  \big(f(  R_0 +(\kk\wedge|q| ))+f(  R_0 -(\kk\wedge|q| ))-2f(  R_0)\big)\le 0,
$$thanks to $\kk<R_0$ and  Lemma \ref{lem3} (iii) below.

On the other hand, by the definition of the constant $R_0$, for all $r\ge R_0$,
\begin{align*}
\Theta_1+\Theta_3&= \vv f(  R_0 )\Big( -\ff{c_0}{2}\big(\mathcal W(x,v)+\mathcal W(x',v')\big)\\
&\quad\qquad\qquad-\ff{c_0}{2}\big(\mathcal W(x,v)+\mathcal W(x',v')\big) +2C_0+2c_* \big(\mathcal W(x,v)^{\eta}+\mathcal W(x',v')^{\eta}\big)\Big)\\
&\le -\ff{c_0\vv}{2}f(  R_0 )\big(\mathcal W(x,v)+\mathcal W(x',v')\big)\\
&\le -\ff{c_0\vv}{1+2\vv}\big(\hat H G \big) \big((x,v),(x',v')\big),
\end{align*} where in the last inequality we used the fact that $$G\big((x,v),(x',v')\big)\le (\vv+1/2)\big(\mathcal W(x,v)+\mathcal W(x',v')\big)$$
by virtue of $\mathcal W\ge1$.

Combining all the estimates above yields the desired assertion.
\end{proof}

\begin{lemma}\label{lem7}
Assume $({\bf A_0})$, $({\bf A_1})$ and $({\bf A_2})$ hold.   Then, for any $r<R_0,$
\begin{equation}\label{E21}
\big(\tilde{\mathscr L}(\hat H G) \big) \big((x,v),(x',v')\big)\le - \ff{3c_1}{8(1+c_1)}(1-1/\aa_0)b\aa  (\hat H G)   \big((x,v),(x',v')\big),
\end{equation} where $\aa_0, c_1>0$ are given in \eqref{E14} and \eqref{E16}, respectively.
\end{lemma}

\begin{proof}
We still adopt the shorthand notation $\Theta_1,\Theta_2,\Theta_3$ introduced in \eqref{B15}. By applying the Young inequality that $ab\le a^p/p+b^q/q$ for all $a,b>0$ and $1/p+1/q=1$ with $p,q>1$, we find that for any $r<R_0,$
\begin{equation}\label{E6}
\begin{split}
\Theta_1+\Theta_3&\le \vv f(r  )\Big(2C_0-c_0\big(\mathcal W(x,v)+\mathcal W(x',v')\big)+2c_* \big(\mathcal W(x,v)^{\eta}+\mathcal W(x',v')^{\eta}\big)\Big)\\
&\le 2 \vv \Big(C_0+(1-\eta)\big(2c_*(\eta/c_0)^{\eta}\big)^{{1}/({1-\eta})}\Big)f(r)\\
&\le2 \vv \Big(C_0+(1-\eta)\big(2c_*(\eta/c_0)^{\eta}\big)^{{1}/({1-\eta})}\Big)\big(\hat H G \big) \big((x,v),(x',v')\big),
\end{split}
\end{equation}
where the last inequality is due to $G\ge1.$

Next, we aim to  estimate $\Theta_2$. Note that, according to ${\bf(A_0)}$, the definition of $\lambda^*(R_0)$ and the fact that $q=z+\aa^{-1}w$,
\begin{align*}
&\aa_0(a-b\aa)|z|+\ff{b\aa\aa_0}{|z|}\< z,q\>
 +\ff{1}{|q|}\big\<q,az+bw+\aa^{-1}(U(x,v)-U(x',v'))\\
 &\le\aa_0(a-b\aa)|z|+ b\aa\aa_0|q|+ |az+bw+\aa^{-1}(U(x,v)-U(x',v'))|\\
 &\le \aa_0(a-b\aa)|z|+ b\aa\aa_0|q|+ (a+\aa^{-1}\lambda^*(R_0))|z|+(b+\aa^{-1}\lambda^*(R_0))|w|\\
&\le \big(a(1+\aa_0)+(1+1/\aa)\lambda^*(R_0)-(\aa_0-1)b\aa\big)|z|
+\big(\lambda^*(R_0)+b\aa(1+\aa_0)\big)|q| \\
&=:\Lambda(|z|,|q|) .
\end{align*}

\noindent{\bf Case: $|z|\ge k_0|q|$}, where $k_0$ is defined in \eqref{E15}. One  has
\begin{equation}\label{E12}
\aa_0|z|\le r \le(\aa_0+1/k_0) |z|
\end{equation}
and
\begin{align*}
\Lambda(|z|,|q|)&\le\big((\lambda^*(R_0)+b\aa(1+\aa_0))/k_0+a(1+\aa_0)+(1+1/\aa)\lambda^*(R_0)-(\aa_0-1)b\aa\big)|z|\\
&=\Big(-\ff{3}{4}(1-1/\aa_0)(\aa_0+1/k_0)b\aa+I_1+I_2\Big)|z|,
\end{align*}
where
\begin{align*}
I_1:&=-\ff{1}{8}(\aa_0-1)b\aa+a(1+\aa_0)+(1+1/\aa)\lambda^*(R_0),\\
I_2:&=-\ff{1}{8}(\aa_0-1)b\aa+\ff{1}{k_0}(\lambda^*(R_0)+b\aa(1+\aa_0))+\ff{3}{4k_0}(1-1/\aa_0)b\aa.
\end{align*}
According to the choices of $\aa$
 and $\aa_0$ given in \eqref{E14}, we get $I_1=0.$ On the other hand, in terms of  the definition of $k_0>0$ defined in \eqref{E15}, we also have $I_2=0.$ Thus,
 $$\Lambda(|z|,|q|) \le\Big(-\ff{3}{4}(1-1/\aa_0)(\aa_0+1/k_0)b\aa\Big)|z|.$$
 Furthermore, in view  of $\kk\wedge |q|\le r<R_0$ and Lemma \ref{lem3}(iii) below,  it follows   that
$$
f(r+(\kk\wedge |q|))+f(r-(\kk\wedge|q|))-2f(r)\le0.
$$ Consequently,  we derive that for $r<R_0$,
\begin{equation}\label{E19}
\begin{split}
\Theta_2&\le -\ff{3}{4}(1-1/\aa_0)(\aa_0+1/k_0)b\aa G\big((x,v),(x',v')\big)f'(r)|z|\\
&\le -\ff{3}{4}(1-1/\aa_0)b\aa G\big((x,v),(x',v')\big)rf'(r)\\
&\le -\ff{3c_1}{4(1+c_1)}(1-1/\aa_0)b\aa \big(\hat H G \big) \big((x,v),(x',v')\big),
\end{split}
\end{equation}
where in the second inequality we used \eqref{E12} and in the third inequality we employed $c_1\le f'(r)\le 1+c_1$ for all $r\in (0,R_0]$ and $f(0)=0.$

\noindent{\bf Case: $|z|\le k_0|q|$}. By means of $\aa_0>1$ and $|q|\le r$, one has
$$
\Lambda (|z|,|q|)\le \Lambda_0\,r,
$$
where $\Lambda_0$ is given in \eqref{E15}.
This, along with Lemma \ref{lem3}(iv), \eqref{W0} and $|q|\le r\le (1+k_0\aa_0)|q|$, yields
\begin{align*}
\Theta_2&\le\ff{1}{2} G\big((x,v),(x',v')\big)\Big(2\Lambda_0f'(r  )r+f''(r)(\kk\wedge |q|)^2\nu^*_{\aa(q)_\kk}(\R^d)\Big)\\
&\le \ff{1}{2} G\big((x,v),(x',v')\big)\Big(2\Lambda_0f'(r  )r+f''(r)J(\aa(\kk\wedge |q|))(\kk\wedge |q|)^2\Big)\\
&\le \ff{1}{2}\Big\{ 2\Lambda_0 f'(r)+\ff{1}{1+k_0\aa_0}f''(r)\si_{\aa,\kappa,R_0}(r/(1+k_0\aa_0)) \Big\}rG\big((x,v),(x',v')\big),
\end{align*}
where in the second inequality we used $f''\le0$ and the fact that $\si_{\aa,\kappa,R_0}$ is a decreasing function.
Since for all $r<R$,
\begin{align*}
f'(r)&=c_1+\e^{-c_2 g(r)},~~~~f''(r)=-c_2\e^{-c_2 g(r)}g'(r),\\
g'(r)&=\ff{C_*}{\si_{\aa,\kappa,R_0}(r/(1+k_0\aa_0)) }
\end{align*}
with $c_1=\e^{-c_2g(2R_0)}$ and  $c_2={3}(1-1/\aa_0)b\aa(1+k_0\aa_0)/2$,
we arrive at
\begin{equation}\label{Y5}
\begin{split}
\Theta_2
&\le \ff{1}{2}\Big\{ 4\Lambda_0   -\ff{c_2}{1+k_0\aa_0}g'(r)\si_{\aa,\kappa,R_0}(r/(1+k_0\aa_0)) \Big\}r\e^{-c_2 g(r)}G\big((x,v),(x',v')\big)\\
&= \ff{1}{2}\Big\{ 4\Lambda_0   - \ff{c_2C_*}{1+k_0\aa_0} \Big\}r\e^{-c_2 g(r)}G\big((x,v),(x',v')\big)\\
&= -\ff{3 }{4 }(1-1/\aa_0)b\aa \,r \e^{-c_2 g(r)}G\big((x,v),(x',v')\big)\\
&\le-\ff{3c_1}{4(1+c_1)}(1-1/\aa_0)b\aa \big(\hat H G \big) \big((x,v),(x',v')\big),
\end{split}
\end{equation}
where in the first inequality we used $c_1\le \e^{-c_2 g(r)}$ for all $r\in (0,R_0]$ and the last inequality is due to $c_1\le \e^{-c_2 g(r)}$ for all $r\in (0,R_0]$ again and $c_1r\le f(r)\le(1+c_1)r.$

Henceforth, the assertion \eqref{E21} follows by
 combining \eqref{E6} with \eqref{E19}  and \eqref{Y5} and taking the alternative of $\vv$, given in \eqref{E18}, into account.
\end{proof}

Putting Lemma \ref{Lemma3.4} and Lemma \ref{lem7} together, we readily obtain the following proposition, which is crucial to establish exponential ergodicity for the process $(X_t,V_t)_{t\ge0}$ determined by \eqref{EE1}.
\begin{proposition}\label{P:proof}
Under assumptions $({\bf A_0})$, $({\bf A_1})$ and $({\bf A_2})$, it holds for all $x,v,x',v'\in \R^d$ that
$$
\big(\tilde{\mathscr L}(\hat H G) \big) \big((x,v),(x',v')\big)\le - \min\Big\{\ff{c_0\vv}{1+2\vv}, \ff{3c_1(1-1/\aa_0)b\aa}{8(1+c_1)} \Big\} (\hat H G)   \big((x,v),(x',v')\big).
$$
\end{proposition}

\section{Proofs of main results}\label{section4}
We begin with the
\begin{proof}[Proof of Theorem $\ref{T:main1}$] Let $P_t((x,v),\cdot)$ be the transition kernel of the process $(X_t,V_t)_{t\ge0}$ starting from $(x,v)\in \R^{2d}$.
First, according to Proposition \ref{P:proof} and \cite[Proposition 4.3]{LWb}, we know that
for any $x,v,x', v'\in \R^d$ and $t>0$,
\begin{equation}\label{e:ooosss}W_{\widetilde\Psi}\big(P_t((x,v),\cdot),P_t((x',v'),\cdot)\big)\le \e^{-\lambda_* t} \widetilde\Psi\big((x,v),(x',v')\big),\end{equation} where
$$
 \widetilde\Psi\big((x,v),(x',v')\big):= (\hat H G)\big((x,v),(x',v')\big)$$ and $$
 \lambda_*:=\min\Big\{\ff{c_0\vv}{1+2\vv}, \ff{3c_1(1-1/\aa_0)b\aa}{8(1+c_1)} \Big\}.
$$

Note that $$\widetilde\Psi\big((x,v),(x',v')\big)\ge \bar\Psi\big((x,v),(x',v')\big):=\bar c \big((|x-x'|+|v-v'|)\wedge 1\big),\quad x,x',v,v'\in \R^d$$ holds with some constant $\bar c>0$. Hence, by \eqref{e:ooosss}, $$W_{\bar\Psi}\big(P_t((x,v),\cdot),P_t((x',v'),\cdot)\big)\le \e^{-\lambda_* t} \widetilde\Psi\big((x,v),(x',v')\big).$$ By this and \cite[Theorem 5.10]{Chen}, we know that for any $t>0$, $P_t$ maps the class of locally Lipschitz continuous functions into itself, where $(P_t)_{t\ge0}$ is the Markov semigroup associated with the process $(X_t,V_t)_{t\ge0}$. Then, with the standard approximation, we
can claim that the process $(X_t, V_t)_{t\ge0}$ is Feller, i.e., for any $t>0$ and $f\in C_b(\R^{2d})$, $P_tf\in C_b(\R^{2d})$.
This, along with Assumption ${\bf(A_1)}$ and \cite[Theorem 4.5]{MT}, yields that the process $(X_t, V_t)_{t\ge0}$ has an invariant probability measure $\mu$ on $\R^{2d} $ such that $\mu(\mathcal W)<\infty$.

Combining these two conclusions above with some more or less standard arguments (see, for example, the proofs
of \cite[Corollary 1.8]{Ma} and \cite[Proposition 1.5]{LW}), we can prove the desired
assertion, also thanks to the fact that there exists a constant $c_0>0$ such that  $$c_0^{-1}\Psi\big((x,v),(x',v')\big)\le \widetilde\Psi\big((x,v),(x',v')\big)\le c_0\Psi\big((x,v),(x',v')\big)$$ for all $x,v,x',v'\in \R^d$.
\end{proof}

 Below we present some explicit sufficient conditions imposed directly on the coefficients of \eqref{E1aa} and the L\'evy measure $\nu$ such that both Assumptions (${\bf A_1}$) and  (${\bf A_2}$) hold true. First, we have
\begin{lemma}\label{Lemma:K}
Suppose that there exist a non-negative function $\mathcal V_0\in C^1(\R^d)$
and constants $r_0\in \R$ and $r>0$ with $|r_0|<r$, $c>0$ and $C\ge0$ so that for all $x,v\in \R^d$,
\begin{equation}\label{F1}
\<r^2x+r_0v+\nabla \mathcal V_0(x),ax+bv\>+\<v+r_0x,U(x,v)\>\le -c\,(\mathcal V_0(x)+|x|^2+|v|^2)+C;
\end{equation}
and that there exists
a constant $\theta\in (0,1]$ such that
\begin{equation}\label{F6}
\int_{\R^d}\big(|u|^2\wedge |u|^\theta\big)\,\nu(\d u)<\8.
\end{equation}
Then, $({\bf A_1})$ holds true  with
\begin{equation}\label{B5}
\mathcal W(x,v):=1+\mathcal V(x,v)^{\theta/2},\quad x,v\in\R^d,
\end{equation}
where $$\mathcal V(x,v):=1+\mathcal V_0(x)+\ff{r^2}{2}|x|^2+\ff{1}{2}|v|^2+r_0\<x,v\>.$$
\end{lemma}

\begin{proof}
Due to  $0\le \mathcal V_0\in C^1(\R^d)$ and $|r_0|< r$, by the Young inequality, one has
\begin{equation}\label{F2}
\begin{split}
1+\mathcal V_0(x)+\frac{r^2-r_0^2}{4}\big(|x|^2+  r^{-2}|v|^2\big)\le& \mathcal V(x,v)
\le  1+\mathcal V_0(x)+   r^2 |x|^2+ |v|^2.\end{split}
\end{equation}
In particular, $\mathcal W:\R^{2d} \to[1,\8)$ is a $C^{1,2}$-function with
 $\mathcal W(x,v)\to\8$ as $|x|+|v|\to\8$.

According to \eqref{EE-},
we find
\begin{align*}
(\mathscr L\mathcal W)(x,v)&=\ff{\theta}{2}\mathcal V(x,v)^{\theta/2-1}\big\{\<ax+bv,\nn_x\mathcal V(x,v)\>+\<U(x,v),\nn_v\mathcal V(x,v)\>\big\}\\
&\quad+\int_{\R^d}\big(\mathcal W(x,v+ u)-\mathcal W(x,v)- \<\nn_v \mathcal W(x,v),u\>\I_{\{|u|\le 1\}}\big)\,\nu(\d u).
\end{align*}
 This, together with \eqref{F1},
$$\nn_x\mathcal V(x,v)=\nabla \mathcal V_0(x)+ r^2x+r_0v,\quad \nn_v\mathcal V(x,v)=v+r_0x$$
and \eqref{F2} as well as $\theta\in (0,1]$, gives
\begin{align*}
(\mathscr L\mathcal W)(x,v)&\le\ff{\theta}{2}\mathcal V(x,v)^{\theta/2-1}\big\{-c\,(\mathcal V_0(x)+|x|^2+|v|^2)+C\big\}+\Pi(x,v)\\
&\le \ff{\theta}{2}\mathcal V(x,v)^{\theta/2-1}\big(-c_3\mathcal V(x,v)+c_4\big)+\Pi(x,v)\\
&\le -c_5\mathcal W(x,v)+c_6+\Pi(x,v)
\end{align*}
for some constants $c_3, c_4, c_5, c_6>0$, where
\begin{align*}\Pi(x,v):&=\int_{\R^d}\big(\mathcal W(x,v+ u)-\mathcal W(x,v)- \<\nn_v \mathcal W(x,v),u\>\I_{\{|u|\le 1\}}\big)\,\nu(\d u)\\
&=\int_{\{|u|\le 1\}}\big(\mathcal W(x,v+u)-\mathcal W(x,v)- \<\nn_v\mathcal W(x,v),u\>\big)\,\nu(\d u)\\
&\quad+\int_{\{|u|> 1\}}\big(\mathcal V(x,v+u)^{\theta/2}-\mathcal V(x,v)^{\theta/2}\big)\,\nu(\d u)\\
&=:\Pi_1(x,v)+\Pi_2(x,v).
\end{align*}

Noting $\|\nabla^2_v \mathcal W\|_\infty<\infty$,   we obtain from the mean value theorem  that for some $c_7>0,$
$$
\Pi_1(x,v)\le\ff{\|\nabla^2_v \mathcal W\|_\infty}{2}\int_{\{|u|\le 1\}}|u|^2\,\nu(\d u)\le c_7.
$$
On the other hand,  by the basic inequality that $a^{\theta/2}-b^{\theta/2}\le (a-b)^{\theta/2}$ for $a\ge b\ge0$, we derive that
\begin{align*}
\Pi_2(x,v)&\le\int_{\{|u|> 1\}}\big( |u|^2/2+(|v| +|r_0|\cdot|x|)|u| \big)^{\theta/2}\,\nu(\d u)\\
&\le c_8 \big(1 + |v| +|x|   \big)^{\theta/2}\int_{\{|u|> 1\}}|u|^\theta\,\nu(\d u)\\
&\le c_9 \mathcal W(x,v)^{1/2}
\end{align*}
for some constants $c_8,c_9>0,$
where  in the last inequality we used \eqref{F6} and \eqref{F2}.

Therefore, combining all the conclusions above, we verify that the assumption (${\bf A_1}$) is satisfied, thanks to   the Young inequality again.
\end{proof}

Next, we prove

\begin{lemma}\label{Lem4} Suppose that $\nu$ satisfies \eqref{F6} for some $\theta\in (0,1]$ and that there are constants $c>0$ and $\theta_0\in (0,\theta/2)$ such that
\begin{equation}\label{e:jjkkk}\nu(\d z)\ge \nu^*(\d z):=\frac{c}{|z|^{d+\theta_0}}
\I_{\{0<z_1\le 1\}}\,\d z,\end{equation} where $z_1$ is the first component of the vector $z\in \R^d$. Then, ${\bf (A_2)}$ holds. More precisely,
${\bf(A_2)}${\rm(i)} holds for $\nu^*$ with some $r_0>0$ and $\sigma_{r_0}(r) =c_0 r^{1-\theta_0}$, and ${\bf(A_2)}${\rm(ii)} is also satisfied for the measure $\nu^*$ and the function $\mathcal W$ given by \eqref{B5} with $\eta=1/2$.
\end{lemma}

\begin{proof}
It is easy to see that $\nu^*\le \nu.$
By \cite[Example 1.2]{LW}, there exist constants $c_0, r_0>0$ such that
$$J(r)\ge c_0 r^{-\theta_0},\quad r\in(0,r_0],$$  where $J$ is defined by \eqref{e:func-nv}. This implies that Assumption ${\bf(A_2)}$(i) holds with $\sigma_{r_0}(r):=c_0 r^{1-\theta_0}$, which is locally integrable on $[0,\8)$.

On the other hand, by taking the explicit expression of $\mathcal W$ given in \eqref{B5} into consideration and using the inequality that $(a+b)^{\theta/2}\le a^{\theta/2}+b^{\theta/2}$ for all $a,b>0$, we find that for all $x,v,u\in \R^d$,
\begin{equation}\label{EEE}
\begin{split}
|\mathcal W(x,v+  u) - \mathcal W(x,v)|
&\le\big|  r_0\<x ,u\> +\<v,u\>+  |u|^2/2\big|^{\theta/2}\\
&\le c_1(1+|x|+|v|)^{\theta/2}(|u|^{\theta/2}+|u|^\theta)\\
&\le c_2 \mathcal W(x,v)^{{1}/{2}}(|u|^{\theta/2}+|u|^\theta)
\end{split}
\end{equation}
with some constants $c_1,c_2>0$, where in the last inequality we used \eqref{F2}. Furthermore, according to
\eqref{e:jjkkk} and \eqref{e:jjkkk} with $\theta_0\in (0,\theta/2)$, it holds that
\begin{equation*}
 \int_{\R^d} \big(|u|^{\theta}+|u|^{\theta/2}\big)\,\nu^*(\d u)<\infty.
 \end{equation*}
This, along with  \eqref{EEE}, further yields
that there is a constant $c_3>0$ such that for all $x,v\in \R^d$,
$$\int_{\R^d}\big|\mathcal W(x,v+  u) - \mathcal W(x,v)\big|\,\nu^*(\d u)\le c_3 \mathcal W(x,v)^{1/2}.$$ In particular, Assumption ${\bf(A_2)}$(ii) holds with $\eta=1/2$.
\end{proof}

According to Lemma \ref{Lemma:K}, Lemma \ref{Lem4} and Theorem \ref{T:main1}, we have the following statement.

\begin{proposition}\label{P:4.3} Under Assumption ${\bf(A_0)}$ and conditions \eqref{F1}, \eqref{F6} and \eqref{e:jjkkk}, the process $(X_t, V_t)_{t\ge0}$ determined by \eqref{EE1} is exponentially ergodic such that \eqref{e:main-1} holds for $\Psi$ defined by \eqref{e:main-2} with  $\mathcal W$   given by \eqref{B5}.
\end{proposition}

Finally, we will present the proof of Theorem \ref{T:main2}. Recall that \eqref{E1} is a special case of the SDE \eqref{E1aa} with $a=0$, $b=1$, $U(x,v)=-\aa v-\bb \nn U_0(x)$.   Obviously, Assumption ${\bf(B_0)}$ implies that ${\bf(A_0)}$ holds. Assumption ${\bf (B_2)}$ is just \eqref{e:jjkkk} in Lemma \ref{Lem4}. Therefore, by Proposition \ref{P:4.3}, in order to prove Theorem \ref{T:main2}, we only need to verify the following lemma.

\begin{lemma}\label{Lem2}
Suppose that \eqref{e:lll} and Assumption ${\bf(B_1)}$ hold. Then, \eqref{F1} holds with  $a=0$, $b=1$, $U(x,v)=-\aa v-\bb \nn U_0(x)$ and $\mathcal V_0(x)=\bb( U_0(x)+\ll_4|x|^2+\ll_5),$ where $\ll_4, \ll_5\ge0$ are given in \eqref{E3---}.
\end{lemma}

\begin{proof}
In view of $\bb>0$ and \eqref{E3---}, $\mathcal V_0(x)=\bb( U_0(x)+\ll_4|x|^2+\ll_5)\ge0$ for all $x\in \R^d$.
Then, for any $r>r_0>0$ and $\vv>0$, it follows from \eqref{E3---1} and the Young inequality  that
\begin{align*}
\Gamma(x,v):&=\<r^2x+r_0v+ \nabla\mathcal V_0(x) ,ax+bv\>+\<v+r_0x,U(x,v)\>\\
&=\<(r^2+2\bb\ll_4)x+ r_0v+ \bb \nabla U_0(x),v\>+\<v+ r_0x,-\aa v-\bb \nn U_0(x)\>\\
&= -( \aa-r_0)|v|^2-\bb r_0\langle x, \nabla U_0(x)\rangle+(r^2+2\bb\ll_4-\aa r_0)\langle x, v\rangle\\
&\le  -( \aa-r_0)|v|^2-\bb r_0 \ll_1|x|^2-\bb r_0 \ll_2U_0(x)+(r^2+2\bb\ll_4-\aa r_0)\langle x, v\rangle +\ll_3\bb r_0\\
&=-( \aa-r_0)|v|^2-\bb r_0 (\ll_1-\ll_2\ll_4)|x|^2 +(r^2+2\bb\ll_4-\aa r_0)\langle x, v\rangle  \\
&\quad- r_0 \ll_2\mathcal V_0(x)+\bb r_0(\ll_3+  \ll_2 \ll_5)\\
&\le -\Big( \aa-r_0-\ff{\vv}{4}(r^2+2\bb\ll_4-\aa r_0)^2\Big)|v|^2-\Big(\bb r_0 (\ll_1-\ll_2\ll_4)-\ff{1}{\vv}\Big)|x|^2\\
&\quad- r_0 \ll_2\mathcal V_0(x)+\bb r_0(\ll_3+  \ll_2 \ll_5).
\end{align*}

Below, we take $r_0=\aa/2$ and
$$\aa/2<r<  (r_0^2/2+\aa_0\sqrt{\bb (\ll_1-\ll_2\ll_4)}-2\bb\lambda_4)^{1/2},$$ which is well defined by \eqref{W3}.
With the
 choices of $r_0$ and $r$ above, we have
$$(r^2+2\bb\ll_4-\aa r_0)^2< 4\bb (\ll_1-\ll_2\ll_4)(\aa-r_0)r_0.
$$ In particular, we can find a constant  $\vv>0$ such that
$$
\aa-r_0-\ff{\vv}{4}(r^2+2\bb\ll_4-\aa r_0)^2>0,\quad\bb r_0 (\ll_1-\ll_2\ll_4)-\ff{1}{\vv}>0.
$$
Therefore, \eqref{F1} is satisfied.
\end{proof}

\section{Appendix}

\subsection{Wasserstein distance}\label{Appendix1}
Let $\Phi$ be a function on $\R^d\times \R^d$
such that $\Phi({\bf 0},{\bf 0})=0$ and $\Phi$ is strictly positive elsewhere. Given two probability measures $\mu_1$ and
$\mu_2$ on $\R^d$, we define the following quantity (which can be called
a Wasserstein-type distance or a Kantorovich distance)
$$
  W_\Phi(\mu_1,\mu_2)=\inf_{\Pi\in \mathscr{C}(\mu_1,\mu_2)}\int_{\R^d\times\R^d} \Phi(x,y)\,\d\Pi(x,y),
$$
where $\mathscr{C}(\mu_1,\mu_2)$ is the collection of all measures on
$\R^d\times\R^d$ having $\mu_1$ and $\mu_2$ as marginals. In
particular, when $\Phi(x,y)=|x-y|^\theta$ with $\theta\in (0,1]$, $W_\Phi$ is just the
standard $L^1$-Wasserstein distance with the metric $(x,y)\mapsto |x-y|^\theta$, which is simply denoted by
$W_1$ when $\theta=1$; on the other hand, when $\Phi(x,y)=\I_{\{x\neq y\}}$, $W_\Phi$
reduces to the total variation distance
$W_\Phi(\mu_1,\mu_2)=\frac{1}{2}\|\mu_1-\mu_2\|_{\var}.$

Note that in applications it is not necessarily to require that $\Phi(x,y)$ is a distance function. Sometime the following type of distance-like function
$$\Phi(x,y)=\Phi_1(x,y) W(x,y),\quad x,y\in \R^d$$ is more applicable, where $\Phi_1(x,y)$ is a distance function and $W(x,y)$ is a strictly positive weighted function fulfilling some growth condition. For example, $\Phi(x,y)=(1\wedge|x-y|)(1+|x|+|y|)^\theta$ for some $\theta\in (0,1]$ as used in Theorem \ref{T:main2}. Then, the associated  Wasserstein-type distance $W_\Phi$ is of the multiplicative form.
The use of
the multiplicative distance $W_\Phi$ is inspired by \cite{HMS}, where
the weak Harris' theorem was
initiated. As mentioned in \cite{HMS}, the distance
of multiplicative form is more practical for SDEs with degenerate
noises or infinite dimensional SDEs, where convergence in terms of
the total variation norm  no longer holds. We note that the multiplicative distance
$W_\Phi$  is merely  a multiplicative semi-metric, since the
triangle inequality may be violated; see \cite[Section 4]{HMS} for
more related details. See also \cite{EGZb,LMW} for the study of exponential ergodicity for diffusions and SDEs with jumps (including McKean-Vlasov type SDEs) in terms of the multiplicative distance $W_\Phi$, respectively.

\subsection{A technical lemma}
For the sake of convenience, let us recall a technical lemma due to \cite[Lemma 2.8]{LMW}.
\begin{lemma}\label{lem3}
Let $l_0>0$, and let $g\in C([0,2l_0])\cap C^3((0,2l_0])$ be such that $g'(s)\ge0$, $g''(s)\le0,$ and $g'''(s)\ge0$ for $s\in(0,2l_0].$ Then, for any $c>0,$ the function
\begin{equation}\label{YY}
f(s):=
\begin{cases}
c\,s+\int_0^s\e^{-g(u)}\d u,&\quad s\in[0,2l_0],\\
f(2l_0)+f'(2l_0)\ff{s-2l_0}{1+s-2l_0},&\quad s>2l_0,
\end{cases}
\end{equation}
satisfies
\begin{itemize}
\item[{\rm (i)}] $f\in C^1((0,\8))$ and $cs\le f(s)\le (c+1)s$ for all $s\in[0,2l_0];$
\item[{\rm(ii)}] for any $s\in(0,2l_0]$,
$$f'(s)\ge0,~~~f''(s)\le0,~~~f^{(3)}(s)\ge0,~~~f^{(4)}(s)\le0;$$
\item[{\rm(iii)}] for any $0\le \dd\le s,$
$$f(s+\dd)+f(s-\dd)-2f(r)\le0;$$
\item[{\rm(iv)}] for any $0\le \dd\le s\le l_0,$
$$f(s+\dd)+f(s-\dd)-2f(r)\le f''(s)\dd^2.$$
\end{itemize}
\end{lemma}

A typical choice of the function $g$ that satisfies the assumption of Lemma \ref{lem3} is
$$g(r)=\int_0^r 1/\sigma(l)\,\d l,$$ where $\sigma \in C([0,2l_0])\cap C^2((0,2l_0])$ is a non-decreasing and concave function such that $\int_{0+}1/\sigma(l)\,\d l<\infty.$
For example, $\sigma(r)=cr^\theta$ with some $c>0$ and $\theta\in (0,1)$.

\subsection{
Comments on Assumption ${\bf (B_1)}$} \label{section5.3}
It is clear that, if $U_0(x)\ge0$ for all $x\in \R^d$ and there are constants $c_1,c_2>0$ such that
$$\langle x, \nabla U_0(x)\rangle\ge c_1|x|^2-c_2,$$ then Assumption ${\bf (B_1)}$ holds trivially with $\ll_1=c_1$, $\ll_3=c_2$ and $\ll_2=\ll_4=\ll_5=0$.

Next, we claim that, if $U_0\in C^1(\R^d)$ satisfies \begin{equation}\label{e:ff99}\liminf_{|x|\to\infty} \frac{U_0(x)}{|x|^2}=\infty\end{equation} and
\begin{equation}\label{e:ff999}\langle x,\nabla U_0(x)\rangle\ge c_3U_0(x)-c_4,\quad x\in \R^d\end{equation} for some $c_3,c_4>0$, then $U_0$ satisfies Assumption ${\bf (B_1)}$. Indeed, by $U_0\in C^1(\R^d)$ and \eqref{e:ff99}, \eqref{E3---} holds with $\lambda_4=0$ and large $\ll_5>0$. In particular, \eqref{W3} is fulfilled. On the other hand, according to \eqref{e:ff99} and \eqref{e:ff999},  \eqref{E3---1} is satisfied with $\ll_1=\ll_2=c_3/2$ and large $\ll_3>0$.

It is obvious that $U_0(x)=c_1(1+|x|^2)^l-c_2|x|^2$ with $l>1$ or $U_0(x)=c_1\e^{(1+|x|^2)^l}-c_2|x|^2$ with $l>0$ for any $c_1,c_2>0$ satisfies \eqref{e:ff99} and \eqref{e:ff999}.

 \ \

\noindent \textbf{Acknowledgements.}
The research of Jianhai Bao is supported by the National Natural Science Foundation of China (Nos.\ 11771326, 12071340 and 11831014).
The research of Jian Wang is supported by the National
Natural Science Foundation of China (Nos.\ 11831014 and 12071076), the Program for Probability and Statistics: Theory and Application (No.\ IRTL1704) and the Program for Innovative Research Team in Science and Technology in Fujian Province University (IRTSTFJ).

 \end{document}